\documentclass[a4paper,10pt
]{amsart}
\usepackage{amsmath, amscd}
\usepackage{amssymb}
\usepackage[alphabetic]{amsrefs}
\usepackage[T1]{fontenc}
\usepackage[all]{xy}
\usepackage{mathtools}
\usepackage{footmisc}
\usepackage{tikz-cd}
\usepackage{enumitem}
\usepackage{relsize} 
\usepackage[bbgreekl]{mathbbol} 
\usepackage{amsfonts}
\usepackage{comment}
\usepackage{stmaryrd}
\usepackage[colorlinks=true, hyperindex, linkcolor=blue, citecolor=, pagebackref=false, urlcolor=blue, linktocpage]{hyperref}
\usepackage[color = blue!20, bordercolor = black, textsize = tiny]{todonotes}
\usepackage{graphicx}
\graphicspath{ {./} }

\DeclareSymbolFontAlphabet{\mathbb}{AMSb} 
\DeclareSymbolFontAlphabet{\mathbbl}{bbold} 

\numberwithin{equation}{section}
\newtheorem{theorem}[subsection]{Theorem} 
\newtheorem{corollary}[subsection]{Corollary}
\newtheorem{lemma}[subsection]{Lemma}
\newtheorem{proposition}[subsection]{Proposition}
\newtheorem{conjecture}[subsection]{Conjecture}
\newtheorem{thmalpha}{Theorem}

\theoremstyle{definition}
\newtheorem{definition}[subsection]{Definition}
\newtheorem{remark}[subsection]{Remark}
\newtheorem{example}[subsection]{Example}

\newtheorem{construction}[subsection]{Construction}

\newcommand{\cal}{\mathcal}

\newcommand{\arxivlink}[1]{\href{http://arxiv.org/abs/#1}{\texttt{arXiv:#1}}}

\title[]{On the residue sequence in \\ Logarithmic Topological Cyclic Homology}
\author{Tommy Lundemo}
\address{Mathematisch Instituut, Universiteit Utrecht, Budapestlaan 6, 3584 CD Utrecht, Netherlands}
\email{t.lundemo@uu.nl}

\begin{document}
\maketitle

\begin{abstract} Algebraic $K$-theory participates in localization sequences, as first established by Quillen. We identify their boundary maps with the \emph{residue} map from logarithmic ${\rm TC}$ under the cyclotomic trace, resolving a conjecture of Rognes. As a consequence, we obtain e.g.\ a trace map from the $K$-theory of periodic complex $K$-theory to the logarithmic ${\rm TC}$ of connective complex $K$-theory. Moreover, we find that there exists a stable $\infty$-category ${\cal C}$ such that ${\rm TC}({\cal C})$ realizes logarithmic ${\rm TC}$ for specific log structures, such as the natural ones on discrete valuation rings, connective complex $K$-theory, and truncated Brown--Peterson spectra. This provides evidence for the existence of a category of modules in the context of logarithmic geometry. We conjecture that such a category can be chosen to reflect additional structure on logarithmic ${\rm TC}$, and we give evidence for this for discrete valuation rings.  
\end{abstract}

\section{Introduction} Historically, the main motivation for the study of invariants built out of topological Hochschild homology (${\rm THH}$), such as topological cyclic homology (${\rm TC}$), has been their intimate connection to algebraic $K$-theory by means of the trace maps $K \to {\rm TC} \to {\rm THH}$ \cite{DGM13}. However, recent developments have revealed that these approximations of algebraic $K$-theory are of great independent interest, for instance in arithmetic geometry \cite{BMS19}. 

Rognes provided a definition of ${\rm THH}$ in the context of logarithmic geometry \cite{Rog09}, and in joint works with Sagave and Schlichtkrull \cites{RSS15, RSS18, RSS26} fruitfully extend the theory to logarithmic structures in the context of structured ring spectra. From an arithmetic point of view, we consider Rognes' definition to be proven correct: The notion of log prismatic cohomology arising from this definition \cite{BLPO23Prism} participates in the expected comparison isomorphisms \cite{BLMP26a}, and recovers a site-theoretic notion considered by Koshikawa \cite{Kos22} and Koshikawa--Yao \cite{KY25} after suitable Nygaard completion \cite{BLMP26b}, in analogy with \cite{BS22}. 

The purpose of this paper is to resolve a (generalization of a) conjecture from Rognes' ICM-address \cite[Conjecture 7.5]{Rog14}, proving the definition of logarithmic ${\rm THH}$ correct also from a $K$-theoretic point of view. As an immediate consequence of this result, there are compatible horizontal cofibration sequences \begin{equation}\label{eq:quillentorss}\begin{tikzcd}K(k) \ar{r} \ar{d}{{\rm trc}} & K({\cal O}_K) \ar{d}{{\rm trc}} \ar{r} & K(K) \ar{d}{{\rm trc}} \ar{r}{\partial} & K(k)[1] \ar{d}{{\rm trc}} \\ {\rm TC}(k) \ar{r} & {\rm TC}({\cal O}_K) \ar{r} & {\rm TC}({\cal O}_K, {\cal O}_K \cap K^\times) \ar{r}{{\rm res}}  & {\rm TC}(k)[1] \end{tikzcd}\end{equation} connected by the cyclotomic trace. Here, ${\cal O}_K$ is a discrete valuation ring with fraction field $K$ and residue field $k$ and  ${\rm TC}({\cal O}_K, {\cal O}_K \cap K^\times)$ is an instance of Rognes' logarithmic ${\rm TC}$. The upper sequence is Quillen's classical localization sequence \cite[Section 5]{Qui73}, while the lower comes from the recent work of Rognes--Sagave--Schlichtkrull \cite[Theorem 1.4]{RSS26}. 

In the context of structured ring spectra, our results imply the existence of a commutative diagram \begin{equation}\label{eq:bmtorss}\begin{tikzcd}K({\Bbb Z}) \ar{r} \ar{d}{{\rm trc}} & K({\rm ku}) \ar{d}{{\rm trc}} \ar{r} & K({\rm KU}) \ar{d}{{\rm trc}} \ar{r}{\partial} & K({\Bbb Z})[1] \ar{d}{{\rm trc}} \\ {\rm TC}({\Bbb Z}) \ar{r} & {\rm TC}({\rm ku}) \ar{r} & {\rm TC}({\rm ku}, \langle \beta \rangle) \ar{r}{{\rm res}}  & {\rm TC}({\Bbb Z})[1]\end{tikzcd}\end{equation} involving the connective and periodic complex $K$-theory spectra ${\rm ku}$ and ${\rm KU}$. The upper sequence is due to Blumberg--Mandell \cite{BM08}, while the lower is again an instance of \cite[Theorem 1.4]{RSS26}. We stress that the trace map to the logarithmic term ${\rm TC}({\rm ku}, \langle \beta \rangle)$ is also a consequence of our results.

It is important to note that these results do \emph{not} overlap with those of Hesselholt--Madsen \cite{HM03} and Blumberg--Mandell \cite{BM20}, who also discuss cofiber sequences in ${\rm TC}$ involving e.g.\ a map ${\rm TC}(k) \to {\rm TC}({\cal O}_K)$. The latter sequences are, by their very definition, compatible with the analogous sequences in $K$-theory under the cyclotomic trace. On the other hand, Rognes' construction is one intrinsic to logarithmic geometry, and the resulting cofiber sequences mirror classical  \emph{residue} sequences involving logarithmic differential forms. One of the main contributions of the present work is the proof of the commutativity of the right-hand squares in the diagrams \eqref{eq:quillentorss} and \eqref{eq:bmtorss}: The boundary map in the localization sequence in $K$-theory corresponds to the residue map under the cyclotomic trace. Consequently, logarithmic ${\rm TC}$ generalizes both the constructions of \cite{HM03} and \cite{BM20}, but carries \emph{a priori} additional structure arising from its log-geometric construction. 

We shall spend the remainder of the introduction giving further context for and precise statements of our main results.

\subsection{Localization in $K$-theory} One fundamental feature of algebraic $K$-theory is its localization property. For example, if $A$ is a regular Noetherian ring and $x$ is an element of $A$ such that $A/x$ is regular, Quillen \cite[Section 5]{Qui73} establishes a long exact sequence \begin{equation}\label{eq:quillenles}\cdots \to K_{n}(A/x) \to K_n(A) \to K_n(A[x^{- 1}]) \to K_{n - 1}(A / x) \to \cdots .\end{equation} From a modern perspective, two ingredients go into its construction. The first is that algebraic $K$-theory is a \emph{localizing invariant}: As a functor from (small and idempotent complete) stable $\infty$-categories ${\rm Cat}_{\infty}^{\rm perf}$ to spectra, it sends exact/Karoubi sequences\footnote{By definition, a \emph{Karoubi sequence} in ${\rm Cat}_{\infty}^{\rm perf}$ is a sequence ${\cal C} \xrightarrow{i} {\cal D} \xrightarrow{p} {\cal E}$ with $i$ fully faithful, $p \circ i$ null, and the induced map ${\cal D} / {\cal C} \xrightarrow{} {\cal E}$ an equivalence after idempotent completion.} to cofiber sequences. The algebraic $K$-groups considered above are recovered as the homotopy groups $K_n(A) := \pi_nK({\rm Perf}(A))$ with ${\rm Perf}(A)$ the category of perfect (equivalently, compact) complexes of $A$-modules. A typical example of a Karoubi sequence is \begin{equation}\label{eq:xnilkaroubi}{\rm Perf}(A)^{x - {\rm nil}} \to {\rm Perf}(A) \to {\rm Perf}(A[x^{-1}]),\end{equation} involving the $x$-nilpotent modules ${\rm Perf}(A)^{x - {\rm nil}}$. The regularity hypotheses ensure that restriction of scalars induces a functor ${\rm Perf}(A / x) \to {\rm Perf}(A)^{x - {\rm nil}}$, and they moreover allow for \emph{d\'evissage}: This functor induces an equivalence on $K$-theory. From this, we obtain a cofibration sequence \begin{equation}\label{eq:kcof}K(A / x) \to K(A) \to K(A[x^{-1}]) \xrightarrow{\partial} K(A / x)[1], \end{equation} which induces \eqref{eq:quillenles} on homotopy groups. 

D\'evissage fails for most other localizing invariants $E$ of interest, such as topological Hochschild and cyclic homology. While Karoubi sequences \eqref{eq:xnilkaroubi} will still produce cofiber sequences for these theories, their practical value is often stunted by the simplicity of the cofiber term: For example, we have ${\rm THH}({\Bbb Z}_{(p)}[p^{-1}]) \simeq {\Bbb Q}$. An example in the context of ${\Bbb E}_{\infty}$-rings is the equivalence ${\rm THH}({\rm ku}[\beta^{-1}]) \simeq {\rm KU}[K({\Bbb Z}, 3)]$ (cf.\ \cite[Theorem 5.21]{Sto20}), which comes with the additional issue that it is not bounded below, and as such does not fit in the Nikolaus--Scholze approach to cyclotomic spectra \cite{NS18}. 

\subsection{Logarithmic differentials} To illustrate how techniques from logarithmic geometry may remedy the failure of d\'evissage for e.g.\ ${\rm THH}$, let us specialize to the case where $A = {\cal O}_K$ is a discrete valuation ring and $x$ is a chosen uniformizer, with residue field $k$ and fraction field $K$. We consider the diagram \[\begin{tikzcd}[column sep = small]K(k) \ar{r} \ar{d}{{\rm tr}}& K({\cal O}_K) \ar{r} \ar{d}{{\rm tr}} & K(K) \ar{d}{{\rm tr}} \ar{r}{\partial} & K(k)[1] \ar{d}{{\rm tr}} \\ {\rm THH}({\rm Perf}({\cal O}_K)^{x - {\rm nil}}) \ar{r} & {\rm THH}({\cal O}_K) \ar{r} & {\rm THH}(K) \ar{r}{\partial} & {\rm THH}({\rm Perf}({\cal O}_K)^{x - {\rm nil}})[1]\end{tikzcd}\] of horizontal cofibration sequences, connected by the Dennis trace (where we implicitly identify $K(k)$ and $K({\rm Perf}({\cal O}_K)^{x - {\rm nil}})$ by d\'evissage). On $\pi_1$, the right-most square is the outer rectangle in the diagram \begin{equation}\label{eq:tracediagram}\begin{tikzcd}K^\times \ar{r}{u \mapsto \nu(u)} \ar{d}{{\rm dlog}} & {\Bbb Z} \ar{d} \\ \Omega^1_{({\cal O}_K, {\cal O}_K \cap K^\times)} \ar{d} \ar{r}{{\rm res}} & k = \pi_0{\rm THH}(k) \ar{d} \\ \Omega^1_K \ar{r} & \pi_0({\rm THH}({\rm Perf}({\cal O}_K)^{x - {\rm nil}})),\end{tikzcd}\end{equation} with $\nu$ the valuation of $K$. The identification of the upper horizontal map is asserted by Quillen \cite[Section 7, Remark 5.17]{Qui73} and proved by Snaith \cite[Theorem 2.2]{Sna80}. The ${\cal O}_K$-module $\Omega^1_{({\cal O}_K, {\cal O}_K \cap K^\times)}$ of \emph{logarithmic} differentials is obtained from the ordinary differentials $\Omega^1_{{\cal O}_K}$ by adjoining a generator ${\rm dlog}(x)$ subject to the relation $dx = x{\rm dlog}(x)$. The map to $\Omega^1_K$ is determined by ${\rm dlog}(x) \mapsto x^{-1}dx$, while that to $k$ is determined by ${\rm res}({\rm dlog}(x)) = 1$ and ${\rm res}(dx) = 0$. That this description indeed characterizes the left-hand vertical map as $\pi_1$ of the trace map $K(K) \to {\rm THH}(K)$ can, for example, be deduced from \cite[Proposition 2.5]{ABG18}. 

\subsection{Logarithmic ${\rm TC}$} The map ${\rm THH}(k) \to {\rm THH}({\rm Perf}({\cal O}_K)^{x - {\rm nil}})$ induced by restriction of scalars along ${\cal O}_K \to k$ is not an equivalence, and the target of this map is indeed somewhat pathological: It is the ${\rm THH}$ of a \emph{non-}connective ring spectrum (cf.\ Remark \ref{rem:devissageweird}). To overcome this, it was proposed by Rognes \cite{Rog09} to only consider the upper square in the diagram \eqref{eq:tracediagram}, replacing ${\rm THH}(K)$ with the logarithmic term ${\rm THH}({\cal O}_K, {\cal O}_K \cap K^\times)$. There is a cofibration sequence \begin{equation}\label{eq:rssseq}{\rm THH}(k) \to {\rm THH}({\cal O}_K) \to {\rm THH}({\cal O}_K, {\cal O}_K \cap K^\times) \xrightarrow{{\rm res}} {\rm THH}(k)[1]\end{equation} of ${\rm THH}({\cal O}_K)$-modules in cyclotomic spectra (cf.\ \cite[Theorem 1.4]{RSS26}), and the map ${\rm res}$ recovers the residue map $\Omega^1_{({\cal O}_K, {\cal O}_K \cap K^\times)} \to k$ on $\pi_1$.  

One advantage of Rognes' construction is its flexibility: It is defined for any \emph{log ring}. For example, the term ${\rm THH}(A, \langle x \rangle)$ is defined and participates in a residue sequence \eqref{eq:rssseq} for any non-zero divisor $x$ in an arbitrary commutative ring $A$. If $A$ is a discrete valuation ring, the example discussed above is recovered by taking $x$ to be any choice of uniformizer. 

The construction of the sequence \eqref{eq:rssseq} makes no reference to ${\rm THH}$ as an invariant of stable $\infty$-categories, and is thus \emph{a priori} distinct from the constructions of Hesselholt--Madsen \cite{HM03} and Blumberg--Mandell \cite{BM20}. In particular, the first map in \eqref{eq:rssseq} is \emph{not} defined by restriction of scalars. Rather, it comes to life as a shifted boundary map of the cofiber sequence involving the last three terms. Consequently, it has remained an open problem whether the sequence \eqref{eq:rssseq} is compatible with \eqref{eq:kcof} under the trace, and whether it is equivalent to the cofibration sequence \[{\rm THH}(k) \to {\rm THH}({\cal O}_K) \to {\rm THH}({\cal O}_K | x) \xrightarrow{\partial} {\rm THH}(k)[1]\] considered in \cite{HM03}. We prove that this is the case. Indeed, for any commutative ring $A$ with $x \in A$ a non-zero divisor, we offer the following general comparison result to the cofibration sequence \[{\rm THH}(A / x) \to {\rm THH}(A) \to {\rm THH}(A | x) \xrightarrow{\partial} {\rm THH}(A / x)[1]\] with the first map defined by restriction of scalars along $A \to A / x$:  

\begin{thmalpha}\label{thm:mainthm} There is an equivalence \[\varphi \colon {\rm THH}(A, \langle x \rangle) \xrightarrow{\simeq} {\rm THH}(A | x)\] of ${\rm THH}(A)$-modules in cyclotomic spectra, and ${\rm res} \simeq \partial \circ \varphi$.
\end{thmalpha}

As a first consequence, we obtain the diagram \eqref{eq:quillentorss}, and its right-hand square provides a spectral lift of the upper square in 
\eqref{eq:tracediagram}. Moreover, it implies the existence of a stable $\infty$-category realizing logarithmic ${\rm TC}$:

\subsection{Logarithmic modules} There is no agreed-upon notion of quasicoherent sheaves in the context of logarithmic geometry. Consequently, there is no obvious category ${\rm Perf}(A, \langle x \rangle)$ in ${\rm Cat}_\infty^{\rm perf}$ that recovers Rognes' ${\rm THH}(A, \langle x \rangle)$ as the ordinary ${\rm THH}({\rm Perf}(A, \langle x \rangle))$. We note that the functor ${\rm Perf}(A/x) \to {\rm Perf}(A)$ is not fully faithful (the obstruction is the failure of $A \to A/x$ to be a homological epimorphism; that is, the higher homotopy groups of the derived tensor product $A/x \otimes_A^{\Bbb L} A/x$), and so there is no Karoubi sequence of the form \[{\rm Perf}(A/x) \to {\rm Perf}(A) \to {\rm Perf}(A, \langle x \rangle).\] Among the many spectacular recent advancements in our understanding of localizing invariants and motives, the very recent work of Ramzi--Sosnilo--Winges \cite{RSW25} is particularly relevant to us: They prove that the universal localizing invariant ${\cal U}_{\rm loc} \colon {\rm Cat}_{\infty}^{\rm perf} \to {\rm Mot}_{\rm loc}$ of \cite{BGT13} is a Dwyer--Kan localization. As Dwyer--Kan localizations are essentially surjective, we may choose ${\rm Perf}(A, \langle x \rangle)$ in ${\rm Cat}_{\infty}^{\rm perf}$ realizing the cofiber of the induced map ${\cal U}_{\rm loc}{\rm Perf}(A/x) \to {\cal U}_{\rm loc}{\rm Perf}(A)$ of localizing motives as ${\cal U}_{\rm loc}({\rm Perf}(A, \langle x \rangle))$. By construction, $E({\rm Perf}(A, \langle x \rangle))$ will realize the cofiber of $E({\rm Perf}(A/x)) \to E({\rm Perf}(A))$ for any localizing invariant $E$. This choice can even be made functorial \cite[Proposition 2.8]{RSW25}. Theorem \ref{thm:mainthm} implies: 

\begin{thmalpha}\label{thm:logmodules} Let $A$ be a commutative ring and let $x \in A$  be a non-zero divisor. There exists a stable $\infty$-category ${\rm Perf}(A, \langle x \rangle)$ and an equivalence \[{\rm THH}(A, \langle x \rangle) \xrightarrow{\simeq} {\rm THH}({\rm Perf}(A, \langle x \rangle))\] of ${\rm THH}(A)$-modules in cyclotomic spectra. \qed
\end{thmalpha}

By d\'evissage, $K({\rm Perf}(A , \langle x \rangle))$ will recover $K(A[x^{- 1}])$ as soon as $A$ and $A/x$ are regular Noetherian. Thus, there is a trace map $K(A[x^{-1}]) \to {\rm TC}(A, \langle x \rangle)$. A multiplicative trace map of this form also appears in the recent work of Park \cite{Par25} using different methods. There are many examples where this ``logarithmic'' $K$-theory does not recover the $K$-theory of the localization: $K({\Bbb Z}[1/p])$ is not equivalent to $K({\rm Perf}({\Bbb Z}, \langle p^n \rangle))$ compatibly with the localization sequences for $n > 1$, since $K({\Bbb F}_p)$ is not equivalent to $K({\Bbb Z}/p^n)$. 

We stress that Theorem \ref{thm:logmodules} is an existence result, and it remains unclear whether there is a choice of ${\rm Perf}(A, \langle x \rangle)$ that reflects e.g.\ the multiplicative structure of logarithmic ${\rm THH}$. This question is the subject of Conjecture \ref{conj:yougottatryright}. We also remark that Clausen already introduced a cone construction for any exact functor in \cite[Section 3.1]{Cla17}, which also gives rise to a choice of ${\rm Perf}(A, \langle x \rangle)$ in Theorem \ref{thm:logmodules}.  

\subsection{Structures on logarithmic ${\rm THH}$} The logarithmic term ${\rm THH}(A, \langle x \rangle)$ is an instance of Rognes' construction \cite[Definition 8.11]{Rog09}, which produces an ${\Bbb E}_{\infty}$-algebra ${\rm THH}(A, M)$ in cyclotomic spectra for any \emph{log ring} $(A, M)$ \cite[Construction 3.9]{BLPO23Prism}. By definition, a \emph{(pre-)log ring} consists of a commutative ring $A$, a commutative monoid $M$, and a map $\alpha \colon M \to (A, \cdot)$ of commutative monoids. Log rings are the basic building blocks for logarithmic schemes, the basic objects of study in log geometry \cite{Kat89}. The example of ${\rm THH}(A, \langle x \rangle)$ is recovered by taking $M$ to be a free commutative monoid on a single generator (that is, the natural numbers), which is sent to $x$. 

Both Theorems \ref{thm:mainthm} and \ref{thm:logmodules} are concerned with \emph{additive} equivalences of cyclotomic spectra. Emboldened by Theorem \ref{thm:logmodules}, we ambitiously predict that a category ${\rm Mod}(A, M)$ can be chosen to reflect the additional structure on ${\rm THH}(A, M)$. This not only includes its ${\Bbb E}_{\infty}$-ring structure, but also the structure of a \emph{log differential graded ring} on its coefficient ring $\pi_*{\rm THH}(A, M)$. We recall this structure in Section \ref{subsec:logthhlogdiff}, which comes to life in this example via a map ${\rm dlog} \colon {\Bbb S}[BM^{\rm gp}] \to {\rm THH}(A, M)$. The following makes reference to Efimov's result extending localizing invariants $F$ defined on ${\rm Cat}_{\infty}^{\rm perf}$ to localizing invariants $F^{\rm cont}$ on ${\rm Cat}_{\rm st}^{\rm dual}$ \cite[Theorem 0.1]{Efi24}:

\begin{conjecture}\label{conj:yougottatryright} For any log ring $(A, M)$, there exists a symmetric monoidal $\infty$-category ${\rm Mod}(A, M)$ in ${\rm Cat}_{\rm st}^{\rm dual}$ satisfying the following properties: 
\begin{enumerate}
\item[(a)] There is an equivalence \[{\rm THH}(A, M) \xrightarrow{\simeq} {\rm THH}^{\rm cont}({\rm Mod}(A, M))\] of ${\Bbb E}_{\infty}$-algebras in cyclotomic spectra; 
\item[(b)] there is a map ${\rm dlog}_K \colon {\Bbb S}[BM] \to K^{\rm cont}({\rm Mod}(A, M))$ such that \[{\Bbb S}[BM] \xrightarrow{{\rm dlog}_K} K^{\rm cont}({\rm Mod}(A, M)) \xrightarrow{{\rm tr}} {\rm THH}^{\rm cont}({\rm Mod}(A, M))\] induces a log differential graded ring structure on $\pi_*{\rm THH}^{\rm cont}({\rm Mod}(A, M))$;
\item[(c)] the induced isomorphism $\pi_*{\rm THH}(A, M) \xrightarrow{\cong} \pi_*{\rm THH}^{\rm cont}({\rm Mod}(A, M))$ is one of log differential graded rings. 
\end{enumerate}
\end{conjecture}

\noindent Pictorially, Conjecture \ref{conj:yougottatryright} predicts the existence of a commutative diagram \[\begin{tikzcd} \vspace{10 mm} & & & K^{\rm cont}({\rm Mod}(A, M)) \ar{d}{\rm tr} \\ {\Bbb S}[BM] \ar[bend left = 3 mm]{rrru}{{\rm dlog}_K} \ar{r}{\simeq} & {\Bbb S}[BM^{\rm gp}] \ar{r}{{\rm dlog}} & {\rm THH}(A, M) \ar{r}{\simeq} & {\rm THH}^{\rm cont}({\rm Mod}(A, M))\end{tikzcd}\] of ${\Bbb E}_{\infty}$-rings. After the first version of the present preprint appeared, Scholze informed the author of a construction of a rigid symmetric monoidal $\infty$-category ${\rm Mod}(A, M)$ which participates in a residue sequence already at the level of localizing motives (cf.\ \cite{Sch26}). We believe that this will likely lead to a positive answer to Conjecture \ref{conj:yougottatryright} for $M$ fine and saturated.  The formulation of the conjecture is motivated by how log differential graded ring structures come to life in \cite{HM03}:

\subsection{Compatibility with the Hesselholt--Madsen construction} For discrete valuation rings ${\cal O}_K$, Hesselholt--Madsen model the cofiber term ${\rm THH}({\cal O}_K | x)$ using topological Hochschild homology of Waldhausen categories. Recall that a Waldhausen category ${\cal C}$ comes equipped with a subcategory of weak equivalences $w{\cal C}$, which we omit from the notation if it equals the usual class of equivalences in ${\cal C}$. In this context, the role of the Karoubi sequence \eqref{eq:xnilkaroubi} is played by the sequence \[{\rm Perf}({\cal O}_K)^{x - {\rm nil}} \to {\rm Perf}({\cal O}_K) \to ({\rm Perf}({\cal O}_K), v{\rm Perf}({\cal O}_K)),\] where $v{\rm Perf}({\cal O}_K)$ consists of those maps that become quasi-isomorphisms after base-change to $K$. Waldhausen's fibration theorem \cite[Theorem 1.6.4]{Wal85} implies that this gives a cofiber sequence in $K$-theory, while his approximation theorem \cite[Theorem 1.6.7]{Wal85} identifies the $K$-theory of the cofiber term with $K(K)$. 

Hesselholt--Madsen \cite[Theorem 1.3.11]{HM03} give a version of the fibration theorem applicable to their variant of ${\rm THH}$ of Waldhausen categories. This enjoys a version of d\'evissage \cite[Theorem 1]{Dun98}\footnote{See Remark \ref{rem:devissageweird} for discussion on this potentially confusing point.} but no approximation theorem, resulting in a cofiber sequence ${\rm THH}(k) \to {\rm THH}({\cal O}_K) \to {\rm THH}({\cal O}_K | x)$ of ${\rm THH}({\cal O}_K)$-modules in cyclotomic spectra. This formulation gives the cofiber term additional structure: For one, it is an ${\Bbb E}_{\infty}$-ring. Moreover, as we recall in Section \ref{subsec:hmrecall}, there is a map \[{\Bbb S}[B\langle x \rangle] \xrightarrow{{\rm dlog}_K} K({\rm Perf}({\cal O}_K), v{\rm Perf}({\cal O}_K)) \xrightarrow{{\rm tr}} {\rm THH}({\cal O}_K | x)\] which induces a log differential graded ring structure on $\pi_*{\rm THH}({\cal O}_K | x)$.  

\begin{thmalpha}\label{thm:hmcomp} Let ${\cal O}_K$ be a complete discrete valuation ring of mixed characteristic $(0, p)$ with perfect residue field $k$. The isomorphism \[\pi_*{\rm THH}({\cal O}_K, \langle x \rangle) \xrightarrow{\cong} \pi_*{\rm THH}({\cal O}_K | x)\] of Theorem \ref{thm:mainthm} is one of log differential graded rings. 
\end{thmalpha}

Theorems \ref{thm:mainthm} and \ref{thm:hmcomp} show that ${\rm THH}({\cal O}_K, \langle x \rangle)$ is a suitable replacement for Hesselholt--Madsen's ${\rm THH}({\cal O}_K | x)$. More generally, we think of ${\rm THH}(A, \langle x \rangle)$ as a multiplicative refinement of ${\rm THH}(A | x)$. Somewhat surprisingly, we do \emph{not} prove that there is an equivalence of ${\Bbb E}_{\infty}$-algebras in cyclotomic spectra inducing the isomorphism in Theorem \ref{thm:hmcomp}; see Remark \ref{rem:awkward} for a discussion.

\subsection{The case of log ring spectra} One can ask analogous questions in the context of the algebraic $K$-theory of structured ring spectra.  Analogously to the situation for ordinary rings, d\'evissage fails for the ${\rm THH}$ of ring spectra, and it is not the case that the fiber term of the localization sequence \[{\rm THH}({\rm Perf}({\rm ku})^{\beta - {\rm nil}}) \to {\rm THH}({\rm Perf}({\rm ku})) \to {\rm THH}({\rm Perf}({\rm KU}))\] identifies with ${\rm THH}({\Bbb Z})$. In logarithmic ${\rm THH}$, however, Rognes--Sagave--Schlichtkrull \cite[Theorem 1.4]{RSS26} establish a cofiber sequence \[{\rm THH}({\Bbb Z}) \xrightarrow{} {\rm THH}({\rm ku}) \xrightarrow{} {\rm THH}({\rm ku}, \langle \beta \rangle) \xrightarrow{{\rm res}} {\rm THH}({\Bbb Z})[1]\] of ${\rm THH}({\rm ku})$-modules in cyclotomic spectra, which is \emph{a priori} different from the cofiber sequence \[{{\rm THH}({\Bbb Z})} \to {\rm THH}({\rm ku}) \to {\rm THH}({\rm ku} | \beta) \xrightarrow{\partial} {\rm THH}({\Bbb Z})[1]\] studied by Blumberg--Mandell \cite{BM20}. For a general even ${\Bbb E}_2$-ring $A$ and homotopy class $x \in \pi_{2d}(A)$, we write ${\rm THH}(A | x)$ for the cofiber of the map ${\rm THH}(A/x) \to {\rm THH}(A)$ defined by restriction of scalars along $A \to A / x$. This is consistent with the notation ${\rm THH}({\cal O}_K | x)$ and ${\rm THH}({\rm ku} | \beta)$ used above. 

The following generalizes \cite[Conjecture 7.5]{Rog14}, as we explain in Remark \ref{rem:rognesconj}. 

\begin{thmalpha}\label{thm:mainthmspec}  Let $A$ be an even ${\Bbb E}_2$-ring and let $x \in \pi_{2d}(A)$ for $d \ge 0$. Then there is an equivalence \[\varphi \colon {\rm THH}(A, \langle x \rangle) \xrightarrow{\simeq} {\rm THH}(A | x)\] of ${\rm THH}(A)$-modules in cyclotomic spectra, and ${\rm res} \simeq \partial \circ \varphi$. 
\end{thmalpha}

The term ${\rm THH}(A, \langle x \rangle)$ is the variant of log ${\rm THH}$ that appears in the recent work of Rognes--Sagave--Schlichtkrull \cite{RSS26}, of which we give a brief overview in Section \ref{subsec:rssrev}. In this context, it is also equivalent to the definition considered by Ausoni--Bay{\i}nd{\i}r--Moulinos \cite{ABM26} as underlying ${\rm THH}(A)$-modules. Under the hypotheses of Theorem \ref{thm:mainthmspec}, this construction still participates in residue sequences, and recovers their previous construction \cites{RSS15, RSS18} when both are defined (e.g.\ in the example $(A, \langle x \rangle) = ({\rm ku}, \langle \beta \rangle)$). One example which is only defined in the model of \cites{RSS15, RSS18} is that of $({\rm ko}, \langle w \rangle)$ with $w \in \pi_8({\rm ko})$ the Bott class; we refer to Remark \ref{rem:koinclude} for discussion on this point. 

The proof of Theorem \ref{thm:logmodules} still applies in this context. In particular, there exists a stable $\infty$-category ${\rm Perf}(A, \langle x \rangle)$ with ${\rm THH}(A, \langle x \rangle) \simeq {\rm THH}({\rm Perf}(A , \langle x \rangle))$.  Similarly to our discussion in the context of discrete rings, we quickly run into examples where $K({\rm Perf}(A, \langle x \rangle)) \not\simeq K(A[x^{-1}])$. For example, the map $K({\rm Perf}({\rm BP}\langle n \rangle, \langle v_n \rangle))  \to K(E(n))$ in the diagram \[\begin{tikzcd}[row sep = small]\vspace{10 mm} & & K(E(n)) \\ K({\rm BP}\langle n - 1 \rangle) \ar{r} \ar{d}{{\rm trc}} & K({\rm BP}\langle n \rangle) \ar{r} \ar{d}{{\rm trc}} & K({\rm Perf}({\rm BP}\langle n \rangle, \langle v_n \rangle)) \ar[swap]{u}{\not\simeq} \ar{d}{{\rm trc}}  \\ {\rm TC}({\rm BP}\langle n - 1 \rangle) \ar{r} & {\rm TC}({\rm BP}\langle n \rangle) \ar{r} & {\rm TC}({\rm BP}\langle n \rangle, \langle v_n \rangle)\end{tikzcd}\] of horizontal cofiber sequences is not an equivalence for $n \ge 2$: This is the main result of Antieau--Barthel--Gepner \cite{ABG18}. 

\subsection{Acknowledgments} The author was first introduced to questions related to Theorem \ref{thm:mainthm} by Christian Schlichtkrull and would like to thank him, as well as Federico Binda, Jack Davies, Jens Hornbostel, Alberto Merici, Doosung Park, Maxime Ramzi, John Rognes, and Steffen Sagave for helpful discussions, and Rognes and Sagave for comments and corrections on a draft. Finally, the author would like to thank an anonymous referee for helpful comments. This research was conducted in the framework of the DFG-funded research training group GRK 2240: \emph{Algebro-Geometric Methods in Algebra, Arithmetic and Topology} and the programme \emph{Equivariant homotopy theory in context} at the Isaac Newton Institute for Mathematical Sciences, Cambridge. This work was supported by EPSRC grant EP/Z000580/1 and the NWO-grant VI.Veni.242.129. 

\subsection{Outline} In Section \ref{sec:logthh}, we introduce necessary background material on graded cyclotomic spectra and their interaction with the variant of log ${\rm THH}$ recently introduced in \cite{RSS26}. Section \ref{sec:logthhloc} is dedicated to the proofs of Theorems \ref{thm:mainthm} and \ref{thm:mainthmspec}, while we prove Theorem \ref{thm:hmcomp} in Section \ref{sec:hmcomp}. 

\section{Logarithmic Topological Hochschild Homology}\label{sec:logthh} We now introduce the necessary background material on the variant of logarithmic ${\rm THH}$ we shall consider here. After recalling the construction of the replete bar construction following \cite{Rog09}, we introduce graded cyclotomic spectra following \cite{AMMN22}. Finally, we discuss the construction of log ${\rm THH}$ appearing in \cite{RSS26}. 

\subsection{The replete bar construction}\label{subsec:repbar} Let $M$ be a discrete commutative monoid. Its \emph{cyclic bar construction} $B^{\rm cyc}(M)$ is the simplicial commutative monoid with $q$-simplices $M^{\oplus (1 + q)}$ with the usual face and degeneracy maps. Since $M$ is commutative, the iterated multiplication maps induce an augmentation $B^{\rm cyc}(M) \to M$. The \emph{replete bar construction} $B^{\rm rep}(M)$ is the pullback \[\begin{tikzcd}B^{\rm rep}(M) \ar{r} \ar{d} & B^{\rm cyc}(M^{\rm gp}) \ar{d} \\ M \ar{r}{\gamma} & M^{\rm gp} \end{tikzcd}\] of simplicial commutative monoids, and it comes with a map $B^{\rm cyc}(M) \to B^{\rm rep}(M)$. Here we have written $\gamma \colon M \to M^{\rm gp}$ for the canonical map from $M$ to its group completion $M^{\rm gp}$. There is a levelwise isomorphism \begin{equation}\label{eq:repiso}M \oplus B_qM^{\rm gp} \xrightarrow{} B^{\rm rep}_q(M), \quad (m, g_1, \dots, g_q) \mapsto (m, \gamma(m)(g_1 \cdots g_q)^{-1}, g_1, \dots, g_q)\end{equation} which gives rise to an isomorphism $M \oplus BM^{\rm gp} \cong B^{\rm rep}(M)$ of simplicial commutative monoids, cf.\ \cite[Lemma 3.17]{Rog09}.  

\subsection{The circle action on the replete bar construction} For any commutative monoid $M$, the replete bar construction participates in a commutative diagram \begin{equation}\label{eq:repsquare}\begin{tikzcd}B^{\rm cyc}(M) \ar{d} \ar{r} & B^{\rm rep}(M) \ar{r} \ar{d} & B^{\rm cyc}(M^{\rm gp}) \ar{d} \\ M \ar{r}{=} & M \ar{r} & M^{\rm gp} \end{tikzcd}\end{equation} of simplicial commutative monoids. For $M = \langle t \rangle \cong {\Bbb N}$ a free commutative monoid on a single generator $t$, we recall the cyclic structure on the terms and maps involved in the upper horizontal composition $B^{\rm cyc}(\langle t \rangle) \to B^{\rm rep}(\langle t \rangle) \to B^{\rm cyc}(\langle t \rangle^{\rm gp})$. We refer to \cite[Propositions 3.20 and 3.21]{Rog09} and the references therein (in particular Hesselholt \cite{Hes96}) for proofs. 

Let us write $B^{\rm cyc}_{\{j\}}(\langle t \rangle)$ for the fiber over $t^j$ of the augmentation $B^{\rm cyc}(\langle t \rangle) \to \langle t \rangle$, and analogously for $B^{\rm rep}(\langle t \rangle)$ and $B^{\rm cyc}(\langle t \rangle^{\rm gp})$. The upper horizontal composite of \eqref{eq:repsquare} decomposes as \begin{equation}\label{eq:decomp1}\coprod_{j \ge 0} B^{\rm cyc}_{\{j \}}(\langle t \rangle) \to \coprod_{j \ge 0} B^{\rm rep}_{\{j\}}(\langle t \rangle) \to \coprod_{j \in {\Bbb Z}} B^{\rm cyc}_{\{j\}}(\langle t \rangle^{\rm gp}).\end{equation} By design, the latter map is an equivalence on each component for $j \ge 0$, while the former is an equivalence for each $j \ge 1$. One can identify \eqref{eq:decomp1} with the diagram \begin{equation}\label{eq:decomp2} \ast \sqcup \coprod_{j \ge 1} S^1/C_j \to \coprod_{j \ge 0} S^1 / C_j \to \coprod_{j \in {\Bbb Z}} S^1 / C_{|j|}\end{equation} of spaces with $S^1$-action, with trivial action on the weight zero component $S^1/C_0$ (despite the $S^1$-action on $B^{\rm rep}_{\{0\}}(\langle t \rangle)$ not being strictly trivial, see \cite[Proof of Proposition 3.21]{Rog09} for a discussion). 

\subsection{Spectra with Frobenius lifts}\label{subsec:froblift} After passing to suspension spectra, the equivalence between \eqref{eq:decomp1} and \eqref{eq:decomp2} is of spectra with $S^1$-action. In order to upgrade this to an equivalence of cyclotomic spectra, we need that the space-level equivalences are of \emph{spaces with Frobenius lifts} in the sense considered in \cite[Section 2.2]{McC24}. This is the case by \cite[Example 2.2.9]{McC24}. Consequently, the equivalence $k \colon B^{\rm rep}_{\{0\}}(\langle t \rangle) \xrightarrow{\simeq} S^1 / C_0$ in weight zero is one of spaces with Frobenius lifts. Inspecting \cite[Proposition 2.2.7, Corollary 2.2.8, Example 2.2.9]{McC24}, we see that this comes to life from a chain of equivalences $B^{\rm cyc}({\Bbb Z}) \xrightarrow{\simeq} L(B{\Bbb Z}) \xleftarrow{\simeq} \coprod_{i \in {\Bbb Z}} S^1 / C_{|i|}$ of spaces with Frobenius lifts, of which $k$ is the zeroth component. This gives rise to an equivalence \[{\Bbb S}[k] \colon {\Bbb S}[B^{\rm rep}_{\{0\}}(\langle t \rangle)] \xrightarrow{\simeq} {\Bbb S}[S^1/C_0]\] of spectra with Frobenius lifts \cite[p.\ 4333]{McC24} that we shall fix throughout. The latter category admits a left adjoint to cyclotomic spectra by \cite[Theorem 2.3.8]{McC24}, informally obtained by only remembering the ``underlying cyclotomic spectrum'' with cyclotomic structure maps $\varphi_p \colon X \xrightarrow{\psi_p} X^{hC_p} \xrightarrow{{\rm can}} X^{tC_p}$. In particular, the maps $* \to S^1 / C_0 \to *$ induce a further equivalence ${\Bbb S}[S^1 / C_0] \xrightarrow{\simeq} {\Bbb S}^{\rm triv} \oplus {\Bbb S}^{\rm triv}[1]$ of spectra with Frobenius lifts, and hence of cyclotomic spectra.  

\subsection{Mapping spaces in module categories} Let ${\cal C}$ be a presentably symmetric monoidal $\infty$-category and let $A$ be an algebra object in ${\cal C}$. We record the following well-known description of mapping spaces in ${\rm Mod}_{{A}}({\cal C})$: 

\begin{lemma}\label{lem:mappingspectramodules} There is an equivalence \[{\rm Map}_{{\rm Mod}_A({\cal C})}(M, N) \xrightarrow{\simeq} \lim_{[n] \in \Delta} ({\rm Map}_{{\cal C}}(M \otimes A^{\otimes n}, N)),\] natural in $M$ and $N$. 
\end{lemma}

\begin{proof} Let us write the $A$-module $M$ as the colimit of the two-sided bar construction $B_{\bullet}(M, A, A)$ in ${\cal C}$. This gives rise to the first equivalence in the composition \[\begin{tikzcd}[row sep = small]{\rm Map}_{{\rm Mod}_A({\cal C})}(M, N) \ar{r}{\simeq} & \lim_{[n] \in \Delta}({\rm Map}_{{\rm Mod}_A({\cal C})}(M \otimes A^{\otimes n} \otimes A, N)) \ar{d}{\simeq}  \\ \vspace{10 mm} & \lim_{[n] \in \Delta}({\rm Map}_{{\cal C}}(M \otimes A^{\otimes n}, N)),\end{tikzcd}\] while the second follows from adjunction.  
\end{proof}

\subsection{Graded cyclotomic spectra} Following \cite[Appendix A]{AMMN22}, we now discuss some background material on graded cyclotomic spectra. We shall write ${\rm GrSp}$ for the category of ${\Bbb Z}$-graded spectra; that is, the functor category ${\rm Fun}({\Bbb Z}, {\rm Sp})$ where ${\Bbb Z}$ is considered a discrete category. We consider the endofunctor \[F_p \colon {\rm GrSp}^{BS^1} \to {\rm GrSp}^{BS^1}, \quad X \mapsto X^{tC_p},\] where $(X^{tC_p})_i = X_{pi}^{tC_p}$. The category of \emph{graded cyclotomic spectra} is the pullback \[\begin{tikzcd}{\rm GrCycSp} \ar{r} \ar{d} & \prod_p {\rm Fun}(\Delta^1, {\rm GrSp}^{BS^1}) \ar{d}{({\rm ev}_0, {\rm ev}_1)} \\ {\rm GrSp}^{BS^1} \ar{r}{({\rm id}, F_p)} & \prod_p {\rm GrSp}^{BS^1}  \times {\rm GrSp}^{BS^1}.  \end{tikzcd} \] If $R$ is a ${\Bbb Z}$-graded ${\Bbb E}_1$-ring, ${\rm THH}(R)$ is naturally an object of ${\rm GrCycSp}$ by \cite[Example A.10]{AMMN22}, in the sense that its underlying cyclotomic spectrum is the ${\rm THH}$ of the underlying ${\Bbb E}_1$-ring of $R$ \cite[Example A.16]{AMMN22}. In this situation, we shall write ${\rm THH}^{\rm gr}(R)$ for the graded cyclotomic spectrum (cf.\ \cite[Section 4.1]{McC24}), so that \cite[Example A.16]{AMMN22} reads as an equivalence ${\rm und}({\rm THH}^{\rm gr}(R)) \simeq {\rm THH}({\rm und}(R))$, where ${\rm und} \colon {\rm GrCycSp} \to {\rm CycSp}$ sends a graded cyclotomic spectrum $X$ to its underlying cyclotomic spectrum $\bigoplus_{n \in {\Bbb Z}} X_n$.  We shall typically simply write ${\rm THH}(R)$ for ${\rm und}({\rm THH}^{\rm gr}(R))$ when $R$ is a graded ${\Bbb E}_1$-ring. We consider ${\rm GrSp}$ as a symmetric monoidal category via Day convolution and use \cite[Construction IV.2.1]{NS18} to obtain a symmetric monoidal structure on ${\rm GrCycSp}$. 

To avoid confusion with the homotopical grading, we shall refer to the grading of objects in ${\rm GrCycSp}$ as the \emph{weight}. The additional structure provided by the weight grading gives us important additional control of mapping spectra in specific cases. We isolate here the particular statement relevant to us:

\begin{lemma}\label{lem:gradedcontrol} Let $R \in {\rm Alg}({\rm GrCycSp})$ be an ${\Bbb E}_1$-algebra in graded cyclotomic spectra concentrated in non-negative weights, and let $M$ and $N$ be $R$-modules in graded cyclotomic spectra. Suppose that 
\begin{enumerate}
\item the weight zero component $R_0$ of $R$ is ${\Bbb S}$;
\item $M$ is concentrated in non-negative weights; and
\item $N$ is concentrated in non-positive weights. 
\end{enumerate}
Then the forgetful functor ${\rm Mod}_R({\rm GrCycSp}) \to {\rm GrCycSp}$ induces an equivalence \[{\rm map}_{{\rm Mod}_R({\rm GrCycSp})}(M, N) \xrightarrow{} {\rm map}_{{\rm GrCycSp}}(M, N)\] of mapping spectra. 
\end{lemma}

\begin{proof} By Lemma \ref{lem:mappingspectramodules}, we obtain the top left-hand vertical equivalence in \begin{equation}\label{eq:compweightdiagram}\begin{tikzcd}[column sep = small]{\rm map}_{{\rm Mod}_R({\rm GrCycSp})}(M, N) \ar{r} \ar{d}{\simeq} & {\rm map}_{{\rm GrCycSp}}(M, N) \ar{d}{\simeq} \\ \lim_{[n] \in \Delta}({\rm map}_{{\rm GrCycSp}}(M \otimes R^{\otimes n}, N)) \ar{d}{\simeq} & \lim_{[n] \in \Delta}({\rm map}_{{\rm GrCycSp}}(M \otimes {\Bbb S}^{\otimes n}, N)) \ar{d}{\simeq}\\ \lim_{[n] \in \Delta}({\rm map}_{{\rm GrCycSp}}(M_0 \otimes R_0^{\otimes n}, N)) \ar{r}{\simeq } &\lim_{[n] \in \Delta}({\rm map}_{{\rm GrCycSp}}(M_0 \otimes {\Bbb S}^{\otimes n}, N)). \end{tikzcd}\end{equation}  The top right-hand vertical map is an equivalence for the same reason (although in this case the limit is simply canonically equivalent to that of the constant diagram ${\rm map}_{{\rm GrCycSp}}(M, N)$). We proceed to explain the remaining vertical equivalences. Since $M$ and $R$ are concentrated in non-negative weights, so is $M \otimes R^{\otimes n}$, and its weight zero piece is $M_0 \otimes R_0^{\otimes n}$. We now argue that the left-hand bottom vertical map is an equivalence. Each level of the limit in the source sits in the pullback \[\begin{tikzpicture}[baseline= (a).base]
\node[scale=.83] (a) at (0,0){\begin{tikzcd}[column sep = tiny]{\rm map}_{{\rm GrCycSp}}(M \otimes R^{\otimes n}, N) \ar{r} \ar{d} & \prod_p {\rm map}_{{\rm Fun}(\Delta^1, {\rm GrSp}^{BS^1})}(M \otimes R^{\otimes n} \to (M \otimes R^{\otimes n})^{tC_p}, N \to N^{tC_p}) \ar{d} \\ {\rm map}_{{\rm GrSp}^{BS^1}}(M \otimes R^{\otimes n}, N) \ar{r} & \prod_p {\rm map}_{{\rm GrSp}^{BS^1}}(M \otimes R^{\otimes n}, N) \times {\rm map}_{{\rm GrSp}^{BS^1}}((M \otimes R^{\otimes n})^{tC_p}, N^{tC_p})\end{tikzcd}};\end{tikzpicture}\] of mapping spectra, while those of the target sit in the pullback square \[\begin{tikzpicture}[baseline= (a).base]
\node[scale=.81] (a) at (0,0){\begin{tikzcd}[column sep = tiny]{\rm map}_{{\rm GrCycSp}}(M_0 \otimes R_0^{\otimes n}, N) \ar{r} \ar{d} & \prod_p {\rm map}_{{\rm Fun}(\Delta^1, {\rm GrSp}^{BS^1})}(M_0 \otimes R_0^{\otimes n} \to (M_0 \otimes R_0^{\otimes n})^{tC_p}, N \to N^{tC_p}) \ar{d} \\ {\rm map}_{{\rm GrSp}^{BS^1}}(M_0 \otimes R_0^{\otimes n}, N) \ar{r} & \prod_p {\rm map}_{{\rm GrSp}^{BS^1}}(M_0 \otimes R_0^{\otimes n}, N) \times {\rm map}_{{\rm GrSp}^{BS^1}}((M_0 \otimes R_0^{\otimes n})^{tC_p}, N^{tC_p}).\end{tikzcd}};\end{tikzpicture}\] The inclusion of the weight zero component $M_0 \otimes R_0^{\otimes n} \to M \otimes R^{\otimes n}$ induces compatible equivalences between each pair of corners of the squares: For example, the induced morphism ${\rm map}_{{\rm GrSp}^{BS^1}}(M \otimes R^{\otimes n}, N) \xrightarrow{} {\rm map}_{{\rm GrSp}^{BS^1}}(M_0 \otimes R_0^{\otimes n}, N)$ is an equivalence, since $M$, $R$, and hence $M \otimes R^{\otimes n}$ are concentrated in non-negative weights, while $N$ is concentrated in non-positive weights. We observe the same is true once the objects are replaced by their $C_p$-Tate constructions, from which the equivalence involving mapping spaces in the arrow category ${\rm Fun}(\Delta^1, {\rm GrSp}^{BS^1})$ follows from the usual ${\rm End}$-description of mapping spaces in functor categories (cf.\ \cite[Proposition 2.3]{Gla16}). We conclude that the lower left-hand vertical map in \eqref{eq:compweightdiagram} is an equivalence, and the same argument applies to the lower right-hand vertical map. The bottom map in \eqref{eq:compweightdiagram} is an equivalence by the assumption $R_0 = {\Bbb S}$. 
\end{proof}

\subsection{The setup of \cite{RSS26}}\label{subsec:rssrev} For an even ${\Bbb E}_2$-ring $A$ with $x \in \pi_{2d}(A)$ for $d \ge 0$, we now follow \cite{RSS26} to give one construction of the term ${\rm THH}(A, \langle x \rangle)$ participating in residue sequences. While \cite{RSS26} provide a general theory for ${\Bbb E}_k$-log rings and their ${\rm THH}$, we restrict attention to the special cases relevant to Theorems \ref{thm:mainthm} and \ref{thm:mainthmspec}. Related material appears in \cite[Section 3.4]{Lur15}, \cite[Section 4.2]{HW22}, and \cite{ABM26}. 

As in \cite{Lur15}, there is a map of ${\Bbb E}_2$-spaces \[\xi_{2d} \colon {\Bbb N} \xrightarrow{\cdot d} {\Bbb Z} \simeq \Omega^2{\rm BU}(1) \to \Omega^2{\rm BU} \simeq {\rm BU} \times {\Bbb Z} \xrightarrow{J_{{\Bbb C}}} {\rm Pic}_{{\Bbb S}}\] with $J_{\Bbb C}$ the complex $J$-homomorphism. The associated Thom spectrum gives the ${\Bbb E}_{2}$-ring ${\Bbb S}[t_{2d}] := {\rm Th}_{{\Bbb S}}(\xi_{2d})$. The class $x$ gives a choice of an ${\Bbb E}_2$-map ${\Bbb S}[t_{2d}] \to A$ by \cite[Corollary 4.8]{RSS26}; this can be achieved by the results of \cite{GKRW25} (see \cite[Proposition 4.6]{RSS26}), much in analogy with \cite[Proposition 3.15]{ABM26}. We here crucially use that $A$ is assumed to be even. The ${\Bbb E}_2$-ring ${\Bbb S}[t_{2d}]$ can be enhanced with a grading, and it is the free ${\Bbb E}_1$-ring on a class $t_{2d}$ of degree $2d$ in weight $1$. We may invert $t_{2d}$ to obtain ${\Bbb S}[t_{2d}^{\pm 1}]$, either by the formalism of \cite[Section 7.2.3]{Lur17} (the Ore condition is vacuous in this case, as $t_{2d}$ is of even degree), or as the Thom spectrum of the ${\Bbb E}_2$-map ${\xi}_{2d}^{\rm gp} \colon {\Bbb Z} \to {\rm Pic}_{{\Bbb S}}$.  

  There is a defining right-hand pullback square \[\begin{tikzcd}B^{\rm cyc}({\Bbb N}) \ar{r} \ar{d} & B^{\rm rep}({\Bbb N}) \ar{d} \ar{r} & B^{\rm cyc}({\Bbb Z}) \ar{d} \ar{r}{B^{\rm cyc}(\xi_{2d}^{\rm gp})} & B^{\rm cyc}({\rm Pic}_{\Bbb S}) \ar{r}{\epsilon} & {\rm Pic}_{\Bbb S} \\ {\Bbb N} \ar{r}{=} & {\Bbb N} \ar{r} & {\Bbb Z}.\end{tikzcd}\] We shall sometimes write \[B^{\rm cyc}({\Bbb N}) \xrightarrow{B^{\rm cyc}(\xi_{2d})} B^{\rm cyc}({\rm Pic}_{{\Bbb S}}) \quad \text{and} \quad B^{\rm rep}({\Bbb N}) \xrightarrow{B^{\rm rep}(\xi_{2d})} B^{\rm cyc}({\rm Pic}_{{\Bbb S}})\] for the resulting maps to $B^{\rm cyc}({\rm Pic}_{\Bbb S})$. We have ${\rm THH}({\Bbb S}[t_{2d}]) \simeq {\rm Th}_{{\Bbb S}}(\epsilon \circ B^{\rm cyc}(\xi_{2d}))$ and ${\rm THH}({\Bbb S}[t_{2d}^{\pm 1}]) \simeq {\rm Th}_{{\Bbb S}}(\epsilon \circ B^{\rm cyc}(\xi_{2d}^{\rm gp}))$ by \cite[Proposition 7.11]{RSS26}. We shall set ${\rm THH}({\Bbb S}[t_{2d}], \langle t_{2d} \rangle) := {\rm THH}({\Bbb S}[t_{2d}^{\pm 1}])_{\ge 0}$ to be the underlying cyclotomic spectrum of the ``weight-connective cover'' ${\rm THH}^{\rm gr}({\Bbb S}[t_{2d}^{\pm 1}])_{\ge 0}$ of ${\rm THH}^{\rm gr}({\Bbb S}[t_{2d}^{\pm 1}])$, and we define \[{\rm THH}(A, \langle x \rangle) := {\rm THH}(A) \otimes_{{\rm THH}({\Bbb S}[t_{2d}])} {\rm THH}({\Bbb S}[t_{2d}], \langle t_{2d} \rangle)\] as a ${\rm THH}(A)$-module in cyclotomic spectra. See \cite[Section 11]{RSS26} in particular for a discussion in this special case. 

\subsection{Discussion of the cyclotomic structure}\label{subsec:cycstr} We continue to follow \cite[Section 11]{RSS26}. The resulting diagram \begin{equation}\label{eq:diagramofgradedcyc}{\rm THH}^{\rm gr}({\Bbb S}[t_{2d}]) \xrightarrow{} {\rm THH}^{\rm gr}({\Bbb S}[t_{2d}^{\pm 1}])_{\ge 0} \xrightarrow{} {\rm THH}^{\rm gr}({\Bbb S}[t_{2d}^{\pm 1}])\end{equation} is one of ${\rm THH}^{\rm gr}({\Bbb S}[t_{2d}])$-modules in graded cyclotomic spectra. The first map is an equivalence in strictly positive weights, and the middle term is underlying equivalent to ${\rm Th}_{\Bbb S}(\epsilon \circ B^{\rm rep}(\xi_{2d}))$ as ${\rm Th}_{{\Bbb S}}(\epsilon \circ B^{\rm cyc}(\xi_{2d})) \simeq {\rm THH}({\Bbb S}[t_{2d}])$-modules. At the level of underlying cyclotomic spectra, the diagram \eqref{eq:diagramofgradedcyc} splits as the inclusions \begin{equation}\label{eq:diagramofgradedcycsplit}\begin{tikzcd}[row sep = small]{\Bbb S}^{\rm triv} \oplus {\rm THH}({\Bbb S}[t_{2d}^{\pm 1}])_{> 0} \ar{d} \\ {\Bbb S}[B^{\rm rep}_{\{0\}}({\Bbb N})] \oplus {\rm THH}({\Bbb S}[t_{2d}^{\pm 1}])_{> 0} \ar{d} \\ {\Bbb S}[B^{\rm rep}_{\{0\}}(\Bbb N)] \oplus {\rm THH}({\Bbb S}[t_{2d}^{\pm 1}])_{> 0} \oplus {\rm THH}({\Bbb S}[t_{2d}^{\pm 1}])_{< 0},\end{tikzcd}\end{equation} as the restriction of $\epsilon \circ B^{\rm rep}(\xi_{2d})$ to the weight zero component $B^{\rm rep}_{\{0\}}({\Bbb N})$ is equivalent to the constant map to ${\rm Pic}_{{\Bbb S}}$, so that its Thom spectrum is the suspension spectrum. 

\begin{example}\label{ex:degreezero} Let us spell out the structures described above when $d = 0$. In this case, ${\rm THH}^{\rm gr}({\Bbb S}[t^{\pm 1}])$ is an ${\Bbb E}_{\infty}$-algebra in ${\rm GrCycSp}$ by \cite[Corollary A.9, Example A.10]{AMMN22}. The splittings \eqref{eq:diagramofgradedcycsplit} now take the explicit form \begin{equation}\label{eq:susprepsplit}\begin{tikzcd}[row sep = small] {\Bbb S}^{\rm triv} \oplus {\Bbb S}[B^{\rm cyc}_{> 0}(\langle t \rangle)] \ar{d} \\  {\Bbb S}[B^{\rm rep}_{\{0\}}(\langle t \rangle)] \oplus {\Bbb S}[B^{\rm cyc}_{> 0}(\langle t \rangle)] \ar{d} \\  {\Bbb S}[B^{\rm rep}_{\{0\}}(\langle t \rangle)] \oplus {\Bbb S}[B^{\rm cyc}_{> 0}(\langle t \rangle)] \oplus {\Bbb S}[B^{\rm cyc}_{< 0}(\langle t \rangle^{\rm gp})]\end{tikzcd}\end{equation} as cyclotomic spectra. 
\end{example}

\section{Logarithmic ${\rm THH}$ and localization}\label{sec:logthhloc}

This section is devoted to the proof of Theorem \ref{thm:mainthmspec}. After introducing our basic setup in Section \ref{subsec:transstudy}, we give an outline of the proof in Section \ref{subsec:proofoutline} that we proceed to realize. 

\subsection{Studying the transfer map}\label{subsec:transstudy} Consider the graded ${\Bbb E}_2$-ring ${\Bbb S}[t_{2d}]$ introduced in Section \ref{subsec:rssrev}. There is a map of ${\Bbb E}_2$-rings $p \colon {\Bbb S}[t_{2d}] \to {\Bbb S}$ obtained by collapsing the strictly positive weights (cf.\ \cite[Lemma B.0.6]{HW22}). It fits in the cofiber sequence \[{\Bbb S}[t_{2d}][2d] \xrightarrow{\cdot t_{2d}} {\Bbb S}[t_{2d}] \xrightarrow{} {\Bbb S}\] of ${\Bbb S}[t_{2d}]$-modules, which exhibits ${\Bbb S}$ as an object of ${\rm Perf}({\Bbb S}[t_{2d}])$. From this, we obtain a transfer map $p_* \colon {\rm Perf}({\Bbb S}) \to {\rm Perf}({\Bbb S}[t_{2d}])$. This sits in the commutative diagram \[\begin{tikzcd}{\rm Perf}({\Bbb S}[t_{2d}])^{t_{2d} - {\rm nil}} \ar{r} & {\rm Perf}({\Bbb S}[t_{2d}]) \ar{r} & {\rm Perf}({\Bbb S}[t_{2d}^{\pm 1}]) \\ {\rm Perf}({\Bbb S}) \ar{u}{p_*^{t_{2d} - {\rm nil}}} \ar{r}{p_*} & {\rm Perf}({\Bbb S}[t_{2d}]), \ar{u}{=}\end{tikzcd}\] where the upper horizontal sequence involving the $t_{2d}$-nilpotent perfect module spectra is Karoubi. Applying ${\rm THH}$ and taking the cofiber of ${\rm THH}(p_*)$, we obtain the diagram of horizontal cofiber sequences \begin{equation}\label{eq:t2dnilcofiberseq}\begin{tikzcd}{\rm THH}({\rm Perf}({\Bbb S}[t_{2d}])^{t_{2d} - {\rm nil}}) \ar{r} & {\rm THH}({\rm Perf}({\Bbb S}[t_{2d}])) \ar{r} & {\rm THH}({\rm Perf}({\Bbb S}[t_{2d}^{\pm 1}])) \\ {\rm THH}({\rm Perf}({\Bbb S})) \ar{u}{{\rm THH}(p_*^{t_{2d} - {\rm nil}})} \ar{r}{{\rm THH}(p_*)} & {\rm THH}({\rm Perf}({\Bbb S}[t_{2d}])) \ar{u}{=} \ar{r} & {\rm THH}({\Bbb S}[t_{2d}] | t_{2d}) \ar{u}{f}\end{tikzcd}\end{equation} in cyclotomic spectra that we shall consider throughout. 

By the discussion in Section \ref{subsec:cycstr}, the map ${\rm THH}({\Bbb S}[t_{2d}]) \to {\rm THH}({\Bbb S}[t_{2d}^{\pm 1}])$ is equivalent to the inclusion \begin{equation}\label{eq:thspecinc}\begin{tikzcd}{\Bbb S}^{\rm triv} \oplus {\rm THH}({\Bbb S}[t_{2d}^{\pm 1}])_{> 0} \ar{r} \ar[bend right = 7 mm]{rdd}  & {\Bbb S}[B^{\rm rep}_{\{0\}}({\Bbb N})] \oplus {\rm THH}({\Bbb S}[t_{2d}^{\pm 1}])_{> 0} \oplus {\rm THH}({\Bbb S}[t_{2d}^{\pm 1}])_{< 0} \ar{d}{{\Bbb S}[k] \oplus {\rm id} \oplus {\rm id}}    \\ \vspace{10 mm} & {\Bbb S}[S^1 / C_0] \oplus {\rm THH}({\Bbb S}[t_{2d}^{\pm 1}])_{> 0} \oplus {\rm THH}({\Bbb S}[t_{2d}^{\pm 1}])_{< 0} \ar{d}{\simeq} \\ \vspace{10 mm} & {\Bbb S}^{\rm triv} \oplus {\Bbb S}^{\rm triv}[1] \oplus {\rm THH}({\Bbb S}[t_{2d}^{\pm 1}])_{> 0} \oplus {\rm THH}({\Bbb S}[t_{2d}^{\pm 1}])_{< 0}\end{tikzcd}\end{equation} Consequently, there is an equivalence \begin{equation}\label{eq:t2dnilsplit}{\rm THH}({\rm Perf}({\Bbb S}[t_{2d}])^{t_{2d} - {\rm nil}}) \simeq {\Bbb S}^{\rm triv} \oplus ({\rm THH}({\Bbb S}[t_{2d}^{\pm 1}])_{< 0})[-1]\end{equation} of cyclotomic spectra determined by the upper cofiber sequence in \eqref{eq:t2dnilcofiberseq}. 

\begin{lemma}\label{lem:nulltransspec} The map  ${\rm THH}({\Bbb S}) \xrightarrow{{\rm THH}(p_*)} {\rm THH}({\Bbb S}[t_{2d}])$ of cyclotomic spectra is null-homotopic. 
\end{lemma}

\begin{proof} The map ${\rm THH}(p_*)$ factors as the composition \[{\rm THH}({\rm Perf}({\Bbb S})) \xrightarrow{{\rm THH}(p_*^{t_{2d} - {\rm nil}})} {\rm THH}({\rm Perf}({\Bbb S}[t_{2d}])^{t_{2d} - {\rm nil}}) \xrightarrow{} {\rm THH}({\rm Perf}({\Bbb S}[t_{2d}])),\] and the latter map is the fiber of the inclusion ${\rm THH}({\Bbb S}[t_{2d}]) \xrightarrow{} {\rm THH}({\Bbb S}[t_{2d}^{\pm 1}])$ of summands in cyclotomic spectra. The desired null-homotopy is, under the splittings \eqref{eq:diagramofgradedcycsplit}, specified by the retraction \[{\Bbb S}[B^{\rm rep}_{\{0\}}({\Bbb N})] \oplus {\rm THH}({\Bbb S}[t_{2d}^{\pm 1}])_{> 0} \oplus {\rm THH}({\Bbb S}[t_{2d}^{\pm 1}])_{< 0} \to {\Bbb S}^{\rm triv} \oplus {\rm THH}({\Bbb S}[t_{2d}^{\pm 1}])_{> 0}\] induced by $B^{\rm rep}_{\{0\}}(\Bbb N) \to *$, ${\rm id}_{{\rm THH}({\Bbb S}[t_{2d}^{\pm 1}])_{> 0}}$, and ${\rm THH}({\Bbb S}[t_{2d}^{\pm 1}])_{< 0} \to 0$. 
\end{proof}

\begin{corollary}\label{cor:undcycspspec} There is an equivalence of cyclotomic spectra \[g \colon {\rm THH}({\Bbb S}[t_{2d}^{\pm 1}])_{\ge 0} \xrightarrow{\simeq} {\rm THH}({\Bbb S}[t_{2d}] | t_{2d})\] under ${\rm THH}({\Bbb S}[t_{2d}])$.    
\end{corollary}

\begin{proof} There is a commutative diagram \[\begin{tikzcd}{\rm THH}({\Bbb S}[t_{2d}]) \ar{r} \ar{d} & {\rm THH}({\Bbb S}[t_{2d}^{\pm 1}])_{\ge 0} \ar{d}  \\ {\Bbb S}^{\rm triv} \ar{r} & {\Bbb S}[B^{\rm rep}_{\{0\}}({\Bbb N})] \end{tikzcd}\] of cyclotomic spectra, with vertical maps projections onto the weight zero components. Their inclusions provide compatible splittings, rewriting the diagram as \[\begin{tikzcd}{\Bbb S}^{\rm triv} \oplus {\rm THH}({\Bbb S}[t_{2d}^{\pm 1}])_{> 0} \ar{r} \ar{d} & {\Bbb S}[B^{\rm rep}_{\{0\}}({\Bbb N})] \oplus {\rm THH}({\Bbb S}[t_{2d}^{\pm 1}])_{> 0} \ar{d} \\ {\Bbb S}^{\rm triv} \ar{r} & {\Bbb S}[B^{\rm rep}_{\{0\}}({\Bbb N})].\end{tikzcd}\]  The desired equivalence arises from the composition \[\begin{tikzcd}{\Bbb S}[B^{\rm rep}_{\{0\}}({\Bbb N})] \oplus {\rm THH}({\Bbb S}[t_{2d}^{\pm 1}])_{> 0} \ar{r}{{\Bbb S}[k] \oplus {\rm id}} \ar[bend right = 8 mm, swap]{ddr}{g}& {\Bbb S}[S^1 / C_0] \oplus {\rm THH}({\Bbb S}[t_{2d}^{\pm 1}])_{> 0} \ar{d}{\simeq} \\ \vspace{10 mm} & {\Bbb S}^{\rm triv} \oplus {\Bbb S}^{\rm triv}[1] \oplus {\rm THH}({\Bbb S}[t_{2d}^{\pm 1}])_{> 0} \ar{d}{\simeq} \\ \vspace{10 mm} & {\rm THH}({\Bbb S}[t_{2d}] | t_{2d})\end{tikzcd}\] with $k$ as in Section \ref{subsec:froblift}. Here, the first vertical equivalence is the canonical splitting, while the second is the splitting specified by the null-homotopy of Lemma \ref{lem:nulltransspec}. 
\end{proof}

We now aim to compare the residue map \[{\rm res} \colon {\rm THH}({\Bbb S}[t_{2d}^{\pm 1}])_{\ge 0} \to {\rm THH}({\Bbb S})[1] \cong {\Bbb S}^{\rm triv}[1]\] appearing in \cite[Proposition 9.9]{RSS26} (where $R = {\Bbb S}$ and $\xi = \xi_{2d}$ in their notation) with the boundary map \[\partial \colon {\rm THH}({\Bbb S}[t_{2d}] | t_{2d}) \to {\rm THH}({\rm Perf}({\Bbb S}))[1] \simeq {\Bbb S}^{\rm triv}[1]\] in the cofiber sequence defining ${\rm THH}({\Bbb S}[t_{2d}] | t_{2d})$. 

\begin{corollary}\label{cor:residuecycl} The maps \[{\rm THH}({\Bbb S}[t_{2d}^{\pm 1}])_{\ge 0} \xrightarrow{{\rm res}} {\Bbb S}^{\rm triv}[1] \text{ and } {\rm THH}({\Bbb S}[t_{2d}^{\pm 1}])_{\ge 0} \xrightarrow{g} {\rm THH}({\Bbb S}[t_{2d}] | t_{2d}) \xrightarrow{\partial} {\Bbb S}^{\rm triv}[1]\] are homotopic in cyclotomic spectra, where $g$ is the equivalence of Corollary \ref{cor:undcycspspec}. 
\end{corollary}

\begin{proof} By \cite[Proposition 9.9]{RSS26}, the map ${\rm res}$ is one of graded cyclotomic spectra. It thus factors as the projection onto its weight zero component ${\Bbb S}[B^{\rm rep}_{\{0\}}({\Bbb N})]$ and a map ${\rm res}_0 \colon {\Bbb S}[B^{\rm rep}_{\{0\}}({\Bbb N})] \to {\Bbb S}^{\rm triv}[1]$. We aim to prove that \[\begin{tikzcd}{\Bbb S}[B^{\rm rep}_{\{0\}}({\Bbb N})] \oplus {\rm THH}({\Bbb S}[t_{2d}^{\pm 1}])_{> 0} \ar{r}{g} \ar{d} & {\rm THH}({\Bbb S}[t_{2d}] | t_{2d}) \ar{d}{\partial} \\ {\Bbb S}[B^{\rm rep}_{\{0\}}({\Bbb N})] \ar{r}{{\rm res}_0} & {\Bbb S}^{\rm triv}[1].\end{tikzcd}\] commutes in cyclotomic spectra. The map ${\rm res}_0$ is specified by the cofiber sequence \[{\Bbb S}^{\rm triv} \cong {\Bbb S}[B^{\rm cyc}_{\{0\}}({\Bbb N})] \xrightarrow{} {\Bbb S}[B^{\rm rep}_{\{0\}}({\Bbb N})] \xrightarrow{{\rm res}_0} {\Bbb S}^{\rm triv}[1]\] of spectra with Frobenius lifts, which we consider a sequence of cyclotomic spectra (this is closely related to \cite[Proposition 9.21]{RSS26}). We now use that the equivalence $g$ of Corollary \ref{cor:undcycspspec} is one under ${\rm THH}({\Bbb S}[t_{2d}])$ to conclude that the diagram commutes, as we get an induced right-hand vertical equivalence in the diagram \[\begin{tikzcd}{\Bbb S}^{\rm triv} \oplus {\rm THH}({\Bbb S}[t_{2d}^{\pm 1}])_{> 0} \ar{r} \ar{d}{\simeq} & {\Bbb S}[B^{\rm rep}_{\{0\}}({\Bbb N})] \oplus {\rm THH}({\Bbb S}[t_{2d}^{\pm 1}])_{> 0} \ar{d}{\simeq} \ar[swap]{d}{g} \ar{r}{\rm res} & {\Bbb S}^{\rm triv}[1] \ar{d}{\simeq} \\ {\rm THH}({\Bbb S}[t_{2d}]) \ar{r} & {\rm THH}({\Bbb S}[t_{2d}] | t_{2d}) \ar{r}{\partial} & {\Bbb S}^{\rm triv}[1]  \end{tikzcd}\] of horizontal cofiber sequences of cyclotomic spectra. The right-hand vertical map is determined up to a sign $\{\pm 1\} \subset {\Bbb Z} \cong \pi_0{\rm TC}({\Bbb S})$ which we can determine to be $1$ by adjusting our conventions for boundary maps: These are fixed so that the boundary map $\pi_1K(K) \to \pi_0K(k)$ in \eqref{eq:quillentorss} sends the chosen uniformizer $\pi$ to $1$ (which is compatible with the log-geometric convention that $\pi_1{\rm THH}({\cal O}_K, \langle \pi \rangle) \to \pi_0{\rm THH}(k)$ sends ${\rm dlog}(\pi)$ to $1$). 
\end{proof}

\subsection{Outline of proof of Theorem \ref{thm:mainthmspec}}\label{subsec:proofoutline} We now pause to give an overview of the ingredients that go into the proof of Theorem \ref{thm:mainthmspec}. The overall strategy is to upgrade the equivalence of Corollary \ref{cor:undcycspspec} to one of ${\rm THH}({\Bbb S}[t_{2d}])$-modules, and to obtain the equivalence of Theorem \ref{thm:mainthmspec}
by base-change. To realize this strategy, we proceed in the following steps:

\begin{itemize}
\item[(a)] Refine the symmetric monoidal functor ${\rm THH} \colon {\rm Cat}_{\infty}^{\rm perf} \to {\rm CycSp}$ to a symmetric monoidal functor ${\rm THH}^{\rm gr} \colon {\rm Mod}_{({\rm GrSp})^{\omega}}({\rm Cat}_{\infty}^{\rm perf}) \to {\rm GrCycSp}$ enjoying the ``graded Morita invariance'' property \[{\rm THH}^{\rm gr}(A) \xrightarrow{\simeq} {\rm THH}^{\rm gr}({\rm Perf}^{\rm gr}(A))\] for a graded ${\Bbb E}_1$-ring $A$ and its perfect modules ${\rm Perf}^{\rm gr}(A)$ (Proposition \ref{prop:thhgr}).  
\item[(b)] Prove that restriction of scalars $p_*^{\rm gr} \colon {\rm Perf}^{\rm gr}({\Bbb S}) \to {\rm Perf}^{\rm gr}({\Bbb S}[t_{2d}])$ is a map of ${\rm Perf}^{\rm gr}({\Bbb S}[t_{2d}])$-modules (Lemma \ref{lem:projectionformulagr}). By (a), we obtain a ${\rm THH}^{\rm gr}({\Bbb S}[t_{2d}])$-linear refinement ${\rm THH}^{\rm gr}(p_*^{\rm gr})$ of ${\rm THH}(p_*)$ to graded cyclotomic spectra, and we may form the cofiber ${\rm THH}^{\rm gr}({\Bbb S}[t_{2d}] | t_{2d})$ in the category of ${\rm THH}^{\rm gr}({\Bbb S}[t_{2d}])$-modules in graded cyclotomic spectra. 
\item[(c)] Use (a) and (b) to show that the map $f \colon {\rm THH}({\Bbb S}[t_{2d}] | t_{2d}) \to {\rm THH}({\Bbb S}[t_{2d}^{\pm 1}])$ of cyclotomic spectra is an inclusion of summands (Proposition \ref{prop:controllingf}).    
\item[(d)] Combine (a) and (b) to upgrade the equivalence of Corollary \ref{cor:undcycspspec} to an equivalence of ${\rm THH}^{\rm gr}({\Bbb S}[t_{2d}])$-modules in graded cyclotomic spectra (Proposition \ref{prop:eqgradedcyc}). This essentially boils down to both maps to ${\rm THH}^{\rm gr}({\Bbb S}[t_{2d}^{\pm 1}])$ being the inclusion of non-negative weights, which follows from (c). 
\item[(e)] Use the additional control of mapping spectra provided by Lemma \ref{lem:gradedcontrol} in the graded setting to upgrade the identification of the residue maps in Corollary \ref{cor:residuecycl} to ${\rm THH}^{\rm gr}({\Bbb S}[t_{2d}])$-modules in graded cyclotomic spectra (Corollary \ref{cor:residuegradedspec}). 
\end{itemize}

Once steps (a)-(e) are established, Theorem \ref{thm:mainthmspec} follows by base-change along ${\rm THH}({\Bbb S}[t_{2d}]) \to {\rm THH}(A)$, as explained in Section \ref{subsec:finalizingproof}.

\subsection{${\rm THH}$ of ${\rm GrSp}$-modules} In order to obtain refinements of Lemma \ref{lem:nulltransspec} and Corollary \ref{cor:undcycspspec} that respect the ${\rm THH}({\Bbb S}[t_{2d}])$-module structures, it will prove important to work in the context of \emph{graded} cyclotomic spectra. For this, we wish to upgrade the functoriality of ${\rm THH}$, usually considered as a symmetric monoidal functor ${\rm THH} \colon {\rm Cat}_{\infty}^{\rm perf} \to {\rm CycSp}$, to a symmetric monoidal functor that respects the graded structures in the examples at hand. For example, for a graded ${\Bbb E}_1$-ring $A$, we let ${\rm Perf}^{\rm gr}(A) := {\rm Mod}_A({\rm GrSp})^\omega$; this is the thick subcategory generated by the grading shifts $A\{q\}$ for $q \in {\Bbb Z}$.

\begin{proposition}\label{prop:thhgr} There is a symmetric monoidal functor \[{\rm THH}^{\rm gr} \colon {\rm Mod}_{({\rm GrSp})^{\omega}}({\rm Cat^{\rm perf}_{\infty}}) \to {\rm GrCycSp}\] with the following properties: 
   \begin{enumerate} \item For a graded ${\Bbb E}_1$-ring $A$, there is a natural equivalence \[{\rm THH}^{\rm gr}(A) \xrightarrow{\simeq} {\rm THH}^{\rm gr}({\rm Perf}^{\rm gr}(A))\] of graded cyclotomic spectra; and
   \item the functor ${\rm und} \colon {\rm GrSp} \to {\rm Sp}, X \mapsto \bigoplus_{n \in {\Bbb Z}} X_n$ gives a commutative square \[\begin{tikzcd}{\rm Mod}_{({\rm GrSp})^{\omega}}({\rm Cat}_{\infty}^{\rm perf}) \ar{r}{{\rm THH}^{\rm gr}} \ar{d}{{\rm und}_!} & {\rm GrCycSp} \ar{d}{\rm und} \\ {\rm Cat}_{\infty}^{\rm perf} \ar{r}{\rm THH} & {\rm CycSp},\end{tikzcd}\] where ${\rm und}_! = {\rm Sp}^\omega \otimes_{({\rm GrSp})^{\omega}} -$ is the base-change along ${\rm und}$. Consequently, for any ${\cal C} \in {\rm Mod}_{({\rm GrSp})^{\omega}}({\rm Cat^{\rm perf}_{\infty}})$, there is a natural equivalence \begin{equation}\label{eq:undeq}{\rm und}({\rm THH}^{\rm gr}({\cal C})) \xrightarrow{\simeq} {\rm THH}({\rm Sp}^{\omega} \otimes_{({\rm GrSp})^\omega} {\cal C})\end{equation} of cyclotomic spectra.
   \end{enumerate}
\end{proposition}

\begin{proof} Objects ${\cal C}$ of ${\rm Mod}_{({\rm GrSp})^{\omega}}({\rm Cat^{\rm perf}_{\infty}})$ are naturally enriched in ${\rm GrSp}$, cf.\ \cite[Section 4.2]{HSS17}. Concretely, the graded mapping spectrum ${\rm map}_{\cal C}^{\rm gr}(X, Y)$ has weight components ${\rm map}_{\cal C}^{\rm gr}(X, Y)_q = {\rm map}_{\cal C}(X\{q\}, Y)$, where the grading shift $X\{q\} := X \otimes {\Bbb S}\{q\}$ uses the $({\rm GrSp})^{\omega}$-module structure. By \cite[Definition 2.11, Theorem 2.14, Proposition 4.23]{HSS17}, we obtain a functor \begin{equation}\label{eq:thhcatperf}{\rm THH}^{\rm gr} \colon {\rm Mod}_{({\rm GrSp})^{\omega}}({\rm Cat^{\rm perf}_{\infty}}) \to {\rm GrSp}^{BS^1}\end{equation} which is symmetric monoidal and agrees with the expected cyclic bar construction and its $S^1$-action by \cite[Proposition 4.24, Remark 6.12]{HSS17}. We upgrade this to a symmetric monoidal functor to ${\rm GrCycSp}$ in Construction \ref{constr:grcycsp}, while properties (1) and (2) are established in Lemma \ref{lem:grthhmorita} and Lemma \ref{lem:undthhgr}, respectively. 
\end{proof}

The remainder of this subsection is devoted to filling in the forward-references used in the above proof of Proposition \ref{prop:thhgr}. Let us first explain the role of \eqref{eq:undeq}:

\begin{example}\label{ex:gradedbasechange} We shall apply Proposition \ref{prop:thhgr}(2) with ${\cal C} = {\rm Perf}^{\rm gr}(A)$. In this case, there is an equivalence \[{\rm Sp}^\omega \otimes_{({\rm GrSp})^\omega} {\rm Perf}^{\rm gr}(A) \simeq {\rm Perf}({\rm und}(A))\] obtained by passing to compact objects in the equivalence \[{\rm Ind}({\rm Sp}^\omega \otimes_{({\rm GrSp})^\omega} {\rm Perf}^{\rm gr}(A)) \simeq {\rm Sp} \otimes_{{\rm GrSp}} {\rm Mod}_A({\rm GrSp}) \simeq {\rm Mod}_{{\rm und}(A)}({\rm Sp})\] arising from symmetric monoidality of ${\rm Ind}$ and \cite[Theorem 4.8.4.6]{Lur17}. The resulting equivalence ${\rm und}({\rm THH}^{\rm gr}({\rm Perf}^{\rm gr}(A))) \xrightarrow{\simeq} {\rm THH}({\rm Perf}({\rm und}(A)))$ is the many-objects analog of \cite[Example A.16]{AMMN22}. While this equivalence could also have been deduced from Proposition \ref{prop:thhgr}(1) and  the ordinary Morita invariance equivalence ${\rm THH}({\rm und}(A)) \xrightarrow{\simeq} {\rm THH}({\rm Perf}({\rm und}(A)))$ for ${\rm THH}$, the additional functoriality obtained from Proposition \ref{prop:thhgr} will be actively used in Section \ref{subsec:esteq}. 
\end{example}

We now upgrade \eqref{eq:thhcatperf} to a functor which lands in graded cyclotomic spectra. The following is essentially a twisted variant of \cite[Construction 2.5]{Mat20}, with some added technicalities since we avoid making reference to point-set models for spectral categories and stay closer to the exposition of \cite{HSS17}. 

\begin{construction}\label{constr:grcycsp} Consider a pair $({\cal C}, S)$ with ${\cal C} \in {\rm Mod}_{({\rm GrSp})^\omega}({\rm Cat}^{\rm perf}_{\infty})$ and $S$ a discrete set, related by a map $S \to {\cal C}^{\simeq}$. Assume moreover that (the images of) $S$ thickly generate ${\cal C}$ up to grading shifts (i.e., thickly generate ${\cal C}$ as a $({\rm GrSp})^\omega$-linear category).\footnote{This is slightly weaker than the condition in \cite[Proposition 4.24]{HSS17}, but the proof still goes through. The proof of Lemma \ref{lem:grthhmorita} gives an instructive example of how this works.} By \cite[Proposition 4.24, Remark 6.12]{HSS17}, the functor \eqref{eq:thhcatperf} applied to ${\cal C}$ can be described as the realization of the cyclic object \begin{equation}\label{eq:qsimplices}[q] \mapsto \bigoplus_{(X_0, \dots, X_q) \in S^{\times (1 + q)}} {\rm map}_{\cal C}^{\rm gr}(X_0, X_1) \otimes \cdots \otimes {\rm map}_{\cal C}^{\rm gr}(X_q, X_0),\end{equation} with the tensor products taken in graded spectra. We denote this cyclic object by $N^{{\rm cy}}_{\rm gr}({\cal C} ; S)_{\bullet}$ and by ${\rm THH}^{\rm gr}({\cal C} ; S)$ its realization, which is naturally an object of ${\rm GrSp}^{BS^1}$. Fixing a summand indexed by $(X_0, \dots, X_q)$ in \eqref{eq:qsimplices}, there is a natural $C_p$-equivariant inclusion into the $p$-fold edgewise subdivision \[({\rm map}_{{\cal C}}^{\rm gr}(X_0, X_1) \otimes \cdots \otimes {\rm map}_{{\cal C}}^{\rm gr}(X_q, X_0))^{\otimes p} \to {\rm sd}_p N^{\rm cy}_{\rm gr}({\cal C} ; S)_q\] and thus an induced map on $C_p$-Tate constructions. Precomposing this with \[\begin{tikzcd}[row sep = small]L_p({\rm map}_{\cal C}^{\rm gr}(X_0, X_1) \otimes \cdots \otimes {\rm map}_{\cal C}^{\rm gr}(X_q, X_0)) \ar{d} \\ (({\rm map}_{\cal C}^{\rm gr}(X_0, X_1) \otimes \cdots \otimes {\rm map}_{\cal C}^{\rm gr}(X_q, X_0))^{\otimes p})^{tC_p}\end{tikzcd}\] as obtained by \cite[Lemma A.6, A.7]{AMMN22} (cf.\ \cite[Example A.10]{AMMN22}; in particular, $L_p$ simply regrades by multiplication by $p$ as in \emph{loc.\ cit.}), we obtain a map $L_p({\rm map}_{\cal C}^{\rm gr}(X_0, X_1) \otimes \cdots \otimes {\rm map}_{\cal C}^{\rm gr}(X_q, X_0)) \to {\rm sd}_p N^{\rm cy}_{\rm gr}({\cal C} ; S)_q^{tC_p}$. Taking the sum of these maps for all $(X_0, \dots, X_q)$, we obtain the map \[L_pN^{\rm cy}_{\rm gr}({\cal C ; S})_q \to {\rm sd}_p N^{\rm cy}_{\rm gr}({\cal C} ; S)_q^{tC_p}.\] The proof of \cite[Proposition A.5]{AMMN22} applies to this construction, so that this assembles to a map of cyclic objects $L_pN^{\rm cy}_{\rm gr}({\cal C} ; S)_{\bullet}\to {\rm sd}_p N^{\rm cy}_{\rm gr}({\cal C} ; S)_\bullet^{tC_p}$,  thus inducing the desired twisted cyclotomic Frobenius \[L_p{\rm THH}^{\rm gr}({\cal C} ; S) \to |{\rm sd}_p N^{\rm cy}_{\rm gr}({\cal C} ; S)_\bullet^{tC_p}| \to |{\rm sd}_p N^{\rm cy}_{\rm gr}({\cal C} ; S)_\bullet|^{tC_p} \to {\rm THH}^{\rm gr}({\cal C} ; S)^{tC_p}\] on realizations.
    
We now let $S$ vary. The category ${\rm GenPair}$ of pairs $({\cal C}, S)$ is the full subcategory of the pullback of the diagram \[{\rm Set} \times {\rm Mod}_{({\rm GrSp})^\omega}({\rm Cat}^{\rm perf}_{\infty}) \xrightarrow{(S, {\cal C}) \mapsto (S, {\cal C}^{\simeq})} {\cal S} \times {\cal S} \xleftarrow{({\rm ev}_0, {\rm ev}_1)} {\rm Fun}(\Delta^1, {\cal S})\] spanned by those $({\cal C}, S)$ where $S$ is $({\rm GrSp})^{\omega}$-linearly thickly generating. This pullback has a natural symmetric monoidal structure by \cite[Construction IV.2.1]{NS18}, from which ${\rm GenPair}$ inherits a symmetric monoidal structure \cite[2.2.1.1, 2.2.1.2]{Lur17}, given by $({\cal C}, S) \otimes ({\cal D}, T) = ({\cal C} \otimes_{({\rm GrSp})^\omega} {\cal D}, S \times T)$. The above construction defines a functor from ${\rm GenPair}$ to ${\rm GrCycSp}$. Moreover, it is symmetric monoidal, since \eqref{eq:thhcatperf} and the twisted Tate diagonal are symmetric monoidal; see \cite[Definition A.1, Proposition A.5, Example A.10]{AMMN22} for the one-object case, and \cite[Construction 2.5]{Mat20} for the many-object version in the ordinary ungraded setting. Define \[{\rm THH}^{\rm gr} \colon {\rm Mod}_{({\rm GrSp})^\omega}({\rm Cat}^{\rm perf}_{\infty}) \to {\rm GrCycSp}\] to be the left Kan extension of $({\cal C}, S) \mapsto {\rm THH}^{\rm gr}({\cal C} ; S)$ along $({\cal C}, S) \mapsto {\cal C}$. This is symmetric monoidal by \cite[Corollary 3.8]{Nik16}, Lemma \ref{lem:quillena} below, and by symmetric monoidality on ${\rm GenPair}$.  
\end{construction}

The following shows that we can always compute ${\rm THH}^{\rm gr}({\cal C})$ by means of the cyclic bar construction \eqref{eq:qsimplices} for any generating set $S$:

\begin{lemma}\label{lem:quillena} For ${\cal C} \in {\rm Mod}_{({\rm GrSp})^\omega}({\rm Cat}^{\rm perf}_{\infty})$, there is a canonical equivalence ${\rm THH}^{\rm gr}({\cal C}) \simeq {\rm THH}^{\rm gr}({\cal C} ; S)$ in ${\rm GrCycSp}$ for any $({\rm GrSp})^\omega$-linearly generating $S$. 
\end{lemma}

\begin{proof} By its definition as a left Kan extension, ${\rm THH}^{\rm gr}({\cal C})$ is computed as the colimit of ${\rm THH}^{\rm gr}({\cal C} ; S)$ over the indexing category appearing as the target of  \begin{equation}\label{eq:quillena}\begin{tikzcd}[row sep = small]{\rm GenPair} \times_{{\rm Mod}_{({\rm GrSp})^\omega}({\rm Cat}^{\rm perf}_{\infty})} \{{\cal C}\} \ar{d} \\ {\rm GenPair} \times_{{\rm Mod}_{({\rm GrSp})^\omega}({\rm Cat}^{\rm perf}_{\infty})} {\rm Mod}_{({\rm GrSp})^\omega}({\rm Cat}^{\rm perf}_{\infty})_{/{\cal C}}.\end{tikzcd}\end{equation} We claim that \eqref{eq:quillena} is cofinal. Aiming to apply Quillen's Theorem A \cite[4.1.3.1]{Lur09}, we find that for fixed $({\cal D}, T)$ in the target of \eqref{eq:quillena}, the relevant comma-category is equivalent to the category whose objects are commutative diagrams \[\begin{tikzcd}T \ar{r} \ar{d} & S \ar{d} \\ {\cal D}^\simeq \ar{r} & {\cal C}^\simeq \end{tikzcd}\] with $S$ the only variable. This category is contractible by the usual contracting-endomorphism trick (e.g.\ \cite[Proof of 6.5.3.7]{Lur09}): For a fixed generating set $G$ of ${\cal C}$, the diagram $S \xrightarrow{} S \sqcup G \xleftarrow{} T \sqcup G$ provides a homotopy between the identity and a constant functor. Thus the Kan extension defining ${\rm THH}^{\rm gr}$ can instead be computed as a colimit over the source of \eqref{eq:quillena}, which is also contractible by the same contracting-endomorphism trick. Thus ${\rm THH}^{\rm gr}({\cal C}) \simeq {\rm THH}^{\rm gr}({\cal C} ; S)$ for any $S$ by \cite[Proposition 4.24]{HSS17}.
\end{proof}

The Morita-invariance property of graded ${\rm THH}$ relevant to us is the following. Strictly speaking, this is an immediate corollary of Lemma \ref{lem:quillena}, but we choose to include a full proof in order to illustrate how to weaken the hypothesis in \cite[Proposition 4.24]{HSS17}, as used in Construction \ref{constr:grcycsp}. 

\begin{lemma}\label{lem:grthhmorita} Let $A$ be a graded ${\Bbb E}_1$-ring. Then the canonical map \[{\rm THH}^{\rm gr}(A) \to {\rm THH}^{\rm gr}({\rm Perf}^{\rm gr}(A))\] is an equivalence of graded cyclotomic spectra.
\end{lemma}

\begin{proof} There is a canonical identification ${\rm THH}^{\rm gr}(A) \simeq {\rm THH}^{\rm gr}({\rm Perf}^{\rm gr}(A) ; \{A\})$ of graded cyclotomic spectra. Although \cite[Proposition 4.24]{HSS17} is not directly applicable to conclude, inspecting its proof we find that we need to prove that the canonical composition map \begin{equation}\label{eq:augeq}\begin{tikzcd}[row sep = small]{|}[q] \mapsto {\rm map}_{{\rm Perf}^{\rm gr}(A)}^{\rm gr}(X, A) \otimes {\rm map}_{{\rm Perf}^{\rm gr}(A)}^{\rm gr}(A, A)^{\otimes q} \otimes {\rm map}_{{\rm Mod}_A({\rm GrSp})}^{\rm gr}(A, M){|} \ar{d} \\ {\rm map}_{{\rm Mod}_A({\rm GrSp})}^{\rm gr}(X, M)\end{tikzcd}\end{equation} is an equivalence for all $X \in {\rm Perf}^{\rm gr}(A)$ and $M \in {\rm Mod}_A({\rm GrSp})$. Indeed, if this is the case, then the proof of \cite[Proposition 4.24]{HSS17} shows that both ${\rm THH}^{\rm gr}(A)$ and ${\rm THH}^{\rm gr}({\rm Perf}^{\rm gr}(A))$ are computed by the trace of the identity of ${\rm Perf}^{\rm gr}(A)$. 
   
In order to prove that \eqref{eq:augeq} is an equivalence, let us first observe that it is an equivalence for $X = A$: Indeed, in this case it is the usual equivalence between the two-sided bar construction $|B_{\bullet}(A, A, M)|$ and $M$. Keeping $M$ fixed and letting $X$ vary, we find that the condition ``\eqref{eq:augeq} is an equivalence'' is clearly closed under cofibers and retracts in the variable $X$. It is also closed under tensoring with objects of $({\rm GrSp})^\omega$: Indeed, for $Y \in ({\rm GrSp})^\omega$, we have equivalences of graded spectra \begin{align*}&{\rm map}_{{\rm Perf}^{\rm gr}(A)}^{\rm gr}(Y \otimes X, A) \simeq Y^\vee \otimes {\rm map}_{{\rm Perf}^{\rm gr}(A)}^{\rm gr}(X, A), \quad \text{ and } \\  &{\rm map}_{{\rm Mod}_A({\rm GrSp})}^{\rm gr}(Y \otimes X, M) \simeq Y^\vee \otimes {\rm map}_{{\rm Mod}_A({\rm GrSp})}^{\rm gr}(X, M).\end{align*} Thus the condition ``\eqref{eq:augeq} is an equivalence'' is satisfied for all grading shifts $A\{q\}$ and the thick subcategory they generate, which is ${\rm Perf}^{\rm gr}(A)$.   
\end{proof}

\begin{remark} It is essential for the above argument that we work with ${\Bbb Z}$-graded and not e.g.\ ${\Bbb N}$-graded spectra. Indeed, compact ${\Bbb N}$-graded spectra do not form a rigid symmetric monoidal category (e.g.\ ${\Bbb S}\{1\}$ has no dual), and so the results of \cite[Section 4.5]{HSS17} do not apply. The above argument is essentially a (specialized) variant of \cite[Theorem 5.12]{BM12}, but following the exposition in \cite{HSS17} allows us to avoid making (implicit) references to point-set models in our construction. 
\end{remark}

We now work towards concluding the proof of Proposition \ref{prop:thhgr}. As in its statement, we consider the symmetric monoidal functor ${\rm und} \colon {\rm GrSp} \to {\rm Sp}$. This preserves compact objects, thus inducing a functor $({\rm GrSp})^{\omega} \to {\rm Sp}^{\omega}$. Below, we also use that it induces a functor ${\rm GrCycSp} \to {\rm CycSp}$ by observing that it is the left Kan extension along ${\Bbb Z} \to *$ \cite[Definition A.13, Lemma A.14]{AMMN22}. 

\begin{lemma}\label{lem:undthhgr} For ${\cal C} \in {\rm Mod}_{({\rm GrSp})^{\omega}}({\rm Cat^{\rm perf}_{\infty}})$, there is a canonical equivalence \[{\rm und}({\rm THH}^{\rm gr}({\cal C})) \xrightarrow{\simeq} {\rm THH}({\rm Sp}^{\omega} \otimes_{({\rm GrSp})^\omega} {\cal C})\] of cyclotomic spectra.
\end{lemma}

\begin{proof} Let $S$ be any set thickly generating ${\cal C}$ up to grading shifts. Consider the canonical comparison map \[{\rm und}({\rm THH}^{\rm gr}({\cal C} ; S)) \to {\rm THH}({\rm Sp}^\omega \otimes_{({\rm GrSp})^\omega} {\cal C} ; {\rm Sp}^{\omega} \otimes_{(\rm GrSp)^\omega} S)\] of cyclotomic spectra, where ${\rm Sp}^{\omega} \otimes_{(\rm GrSp)^\omega} S$ is the image of $S$ under the base-change by ${\rm und}$. Since this thickly generates ${\rm Sp}^\omega \otimes_{({\rm GrSp})^\omega} {\cal C}$, it suffices to prove that the above comparison map is an equivalence by Lemma \ref{lem:quillena} and \cite[Theorem 5.12]{BM12}. This can be checked on underlying $S^1$-equivariant spectra, and so we can use the trace-description of \cite[Proposition 4.24]{HSS17}. Here, the desired result becomes a formal consequence of the results established in \cite[Sections 2 and 4]{HSS17}, which we now unpack. As explained in \cite[Section 4.4]{HSS17}, symmetric monoidal functors such as ${\rm und} \colon ({\rm GrSp})^\omega \to {\rm Sp}^\omega$ give rise to symmetric monoidal functors at the level of Morita $(\infty, 2)$-categories; in this case, a functor \[{\rm und}_{!} \colon {\rm Cat}^{\rm Mor}(({\rm GrSp})^\omega) \to {\rm Cat}^{\rm Mor}({\rm Sp}^\omega).\] The functoriality of traces as in \cite[Theorem 1.7]{HSS17} gives a natural equivalence \[\Omega({\rm und}_!) \circ {\rm Tr}_{{\rm Cat}^{\rm Mor}(({\rm GrSp})^\omega)} \simeq {\rm Tr}_{{\rm Cat}^{\rm Mor}({\rm Sp}^\omega)} \circ {\rm End}({\rm und}_!)\] of functors ${\rm End}({\rm Cat}^{\rm Mor}(({\rm GrSp})^\omega)) \to \Omega {\rm Cat}^{\rm Mor}({\rm Sp}^{\omega}) \simeq {\rm Sp}$, with $\Omega$ the endomorphism category of the unit. Evaluate this equivalence at $({\rm Ind}({\cal C}), {\rm id}_{{\rm Ind}({\cal C})})$ to obtain \[{\rm und}({\rm Tr}_{{\rm Cat}^{\rm Mor}((\rm GrSp)^\omega)}({\rm id}_{{\rm Ind}({\cal C})})) \simeq {\rm Tr}_{{\rm Cat}^{\rm Mor}({\rm Sp}^\omega)}({\rm id}_{{\rm und}_!{\rm Ind}({\cal C})}),\] where we have used that ${\rm und} \simeq \Omega({\rm und}_{!})$, see e.g.\ \cite[Sections 4.5, 6.1]{HSS17}. Since \[{\rm und}_!{\rm Ind}({\cal C}) \simeq {\rm Sp} \otimes_{{\rm GrSp}} {\rm Ind}(\cal C) \simeq {\rm Ind}({\rm Sp}^\omega \otimes_{({\rm GrSp})^\omega} {\cal C})\] by construction \cite[Section 4.4]{HSS17} and symmetric monoidality of ${\rm Ind}$, respectively, this is precisely the desired result at the level of ($S^1$-equivariant \cite[Theorem 1.8]{HSS17}) spectra, since these traces identify with the usual definition of ${\rm THH}$ by \cite[Proposition 4.24, Remark 6.12]{HSS17}.
\end{proof}

\subsection{The transfer is a module map} Let $\varphi \colon R \to S$ be a map of graded ${\Bbb E}_2$-rings, and let $\varphi^* \colon {\rm Perf}^{\rm gr}(R) \to {\rm Perf}^{\rm gr}(S)$ denote the extension of scalars functor, given by $\varphi^*(M) = M \otimes_R S$. The category ${\rm Perf}^{\rm gr}(S)$ is a ${\rm Perf}^{\rm gr}(R)$-module. Informally, the module structure is induced by the composition \[{\rm Perf}^{\rm gr}(S) \otimes {\rm Perf}^{\rm gr}(R) \xrightarrow{{\rm id} \otimes \varphi^*} {\rm Perf}^{\rm gr}(S) \otimes {\rm Perf}^{\rm gr}(S) \xrightarrow{\otimes_S} {\rm Perf}^{\rm gr}(S).\] If the map $\varphi$ moreover exhibits $S$ as an object of ${\rm Perf}^{\rm gr}(R)$, we obtain an adjunction \[\varphi^* \colon {\rm Perf}^{\rm gr}(R) \rightleftarrows {\rm Perf}^{\rm gr}(S) \colon \varphi_*\] with $\varphi^*$ ${\Bbb E}_1$-monoidal and $\varphi_*$ lax ${\Bbb E}_1$-monoidal. 

\begin{lemma}\label{lem:projectionformulagr} The map $\varphi_* \colon {\rm Perf}^{\rm gr}(S) \to {\rm Perf}^{\rm gr}(R)$ is a map of ${\rm Perf}^{\rm gr}(R)$-modules. 
\end{lemma}

\begin{proof} By \cite[7.3.2.8, 7.3.2.9]{Lur17}, it suffices to provide an equivalence \begin{equation}\label{eq:comparisonmappf}\varphi_*(M) \otimes_R N \xrightarrow{\simeq} \varphi_*(M \otimes_S \varphi^*(N)),\end{equation} natural in $M \in {\rm Perf}^{\rm gr}(S)$ and $N \in {\rm Perf}^{\rm gr}(R)$. The map \eqref{eq:comparisonmappf} is defined to be the adjoint of the composition \[\varphi^*(\varphi_*(M) \otimes_R N) \simeq \varphi^*(\varphi_*(M)) \otimes_S \varphi^*(N) \xrightarrow{\epsilon \otimes {\rm id}} M \otimes_S \varphi^*(N)\] obtained from monoidality and the counit of the adjunction. In other words, \eqref{eq:comparisonmappf} being an equivalence is the projection formula, for which the usual proof goes through: \eqref{eq:comparisonmappf} is an equivalence for $N = R\{q\}$ for all $q \in {\Bbb Z}$, and since the condition ``\eqref{eq:comparisonmappf} is an equivalence for $N$'' is closed under cofibers and retracts, it follows that \eqref{eq:comparisonmappf} is an equivalence for all $N \in {\rm Perf}^{\rm gr}(R)$. 
\end{proof} 

\begin{remark} Despite Lemma \ref{lem:nulltransspec}, it is \emph{not} the case that ${\rm THH}(p_*)$ is null-homotopic as a ${\rm THH}({\Bbb S}[t_{2d}])$-module map. If this had been the case, then one could show that e.g.\ the transfer map ${\rm THH}(p_*) \colon {\rm THH}(A / x) \to {\rm THH}(A)$ is null for any commutative ring $A$ and non-zero divisor $x \in A$. But this is false, as for instance seen by Hesselholt--Madsen's computations \cite[Proof of Theorem 2.4.1]{HM03}. 

The author found it reassuring to keep the following elementary example in mind: For an ordinary commutative ring $R$, the sequence \[0 \to R[t] \xrightarrow{t \cdot} R[t] \xrightarrow{t \mapsto 0} R \to 0\] splits as $R$-modules, but not as $R[t]$-modules. In particular, there are non-zero maps $R \to R[t][1]$ in $D(R[t])$, corresponding to ${\rm Ext}^1_{R[t]}(R, R[t]) \cong R$. These all restrict to the zero map in $D(R)$, since ${\rm Ext}^1_R(R, R[t]) \cong 0$. 
\end{remark}

\subsection{Graded refinements} Let us now update our setup in light of Proposition \ref{prop:thhgr}. The cofiber sequence \[{\Bbb S}[t_{2d}][2d]\{1\} \xrightarrow{t_{2d} \cdot} {\Bbb S}[t_{2d}] \to {\Bbb S}\] in ${\rm Mod}_{{\Bbb S}[t_{2d}]}({\rm GrSp})$ exhibits ${\Bbb S}$ as an object of ${\rm Perf}^{\rm gr}({\Bbb S}[t_{2d}])$. We thus obtain a transfer/restriction-of-scalars functor \[p_*^{\rm gr} \colon {\rm Perf}^{\rm gr}({\Bbb S}) \to {\rm Perf}^{\rm gr}({\Bbb S}[t_{2d}]).\] By Proposition \ref{prop:thhgr} and Lemma \ref{lem:projectionformulagr}, this gives a refinement \begin{equation}\label{eq:grtrans}{\rm THH}^{\rm gr}(p_*^{\rm gr}) \colon {\rm THH}^{\rm gr}({\rm Perf}^{\rm gr}({\Bbb S})) \to {\rm THH}^{\rm gr}({\rm Perf}^{\rm gr}({\Bbb S}[t_{2d}]))\end{equation} of the cyclotomic map ${\rm THH}(p_*)$ studied in Section \ref{subsec:transstudy} to a map of ${\rm THH}^{\rm gr}({\Bbb S}[t_{2d}])$-modules in graded cyclotomic spectra. By Example \ref{ex:gradedbasechange}, ${\rm und}({\rm THH}^{\rm gr}(p_*^{\rm gr}))$ is a ${\rm THH}({\Bbb S}[t_{2d}])$-linear refinement of ${\rm THH}(p_*)$. We shall write \[{\rm THH}^{\rm gr}({\Bbb S}[t_{2d}] | t_{2d}) := {\rm cof}({\rm THH}^{\rm gr}(p_*^{\rm gr})),\] where the cofiber is taken in ${\rm Mod}_{{\rm THH}^{\rm gr}({\Bbb S}[t_{2d}])}({\rm GrCycSp})$. Since the composition \[{\rm THH}^{\rm gr}({\rm Perf}^{\rm gr}({\Bbb S})) \xrightarrow{{\rm THH}^{\rm gr}(p_*^{\rm gr})} {\rm THH}^{\rm gr}({\rm Perf}^{\rm gr}({\Bbb S}[t_{2d}])) \to {\rm THH}^{\rm gr}({\rm Perf}^{\rm gr}({\Bbb S}[t_{2d}^{\pm 1}]))\] is null-homotopic, we obtain a map \[f^{\rm gr} \colon {\rm THH}^{\rm gr}({\Bbb S[t_{2d}] | t_{2d}}) \to {\rm THH}^{\rm gr}({\rm Perf}^{\rm gr}({\Bbb S}[t_{2d}^{\pm 1}]))\] in ${\rm Mod}_{{\rm THH}^{\rm gr}({\Bbb S}[t_{2d}])}({\rm GrCycSp})$ refining the cyclotomic map $f$ from \eqref{eq:t2dnilcofiberseq}. In conclusion, there is a diagram \begin{equation}\label{eq:t2dnilcofiberseqgr}\begin{tikzcd}[column sep = tiny]{\rm THH}^{\rm gr}({\rm Perf}^{\rm gr}({\Bbb S}[t_{2d}])^{t_{2d} - {\rm nil}}) \ar{r} & {\rm THH}^{\rm gr}({\rm Perf}^{\rm gr}({\Bbb S}[t_{2d}])) \ar{r} & {\rm THH}^{\rm gr}({\rm Perf}^{\rm gr}({\Bbb S}[t_{2d}^{\pm 1}])) \\ {\rm THH}^{\rm gr}({\rm Perf}^{\rm gr}({\Bbb S})) \ar{u}{{\rm THH}^{\rm gr}({(p_*^{\rm gr}})^{t_{2d} - {\rm nil}})} \ar{r}{{\rm THH}^{\rm gr}(p_*^{\rm gr})} & {\rm THH}^{\rm gr}({\rm Perf}^{\rm gr}({\Bbb S}[t_{2d}])) \ar{u}{=} \ar{r} & {\rm THH}^{\rm gr}({\Bbb S}[t_{2d}] | t_{2d}) \ar{u}{f^{\rm gr}}\end{tikzcd}\end{equation} of horizontal cofiber sequences in ${\rm Mod}_{{\rm THH}^{\rm gr}({\Bbb S}[t_{2d}])}({\rm GrCycSp})$ refining \eqref{eq:t2dnilcofiberseq}. For our purposes one may simply define ${\rm THH}^{\rm gr}({\rm Perf}^{\rm gr}({\Bbb S}[t_{2d}])^{t_{2d} - {\rm nil}})$ to sit in the above cofiber sequence and ${\rm THH}^{\rm gr}({(p_*^{\rm gr}})^{t_{2d} - {\rm nil}})$ for the resulting map of fibers, though one readily shows that the above fiber term is indeed the ${\rm THH}^{\rm gr}$ of the $t_{2d}$-nilpotent perfect graded modules ${\rm Perf}^{\rm gr}({\Bbb S}[t_{2d}])^{t_{2d} - {\rm nil}}$ using \cite[Theorem 3.4]{HSS17}.
 
\subsection{Controlling the map $f \colon {\rm THH}({\Bbb S}[t_{2d}] | t_{2d}) \xrightarrow{} {\rm THH}({\Bbb S}[t_{2d}^{\pm 1}])$} We now aim to explicitly describe the right-hand vertical map in \eqref{eq:t2dnilcofiberseq}:

\begin{proposition}\label{prop:controllingf} Under the splittings fixed in Subsection \ref{subsec:transstudy}, the cyclotomic map $f \colon {\rm THH}({\Bbb S}[t_{2d}] | t_{2d}) \xrightarrow{} {\rm THH}({\Bbb S}[t_{2d}^{\pm 1}])$ is equivalent to the inclusion of summands \[{\Bbb S}^{\rm triv} \oplus {\Bbb S}^{\rm triv}[1] \oplus {\rm THH}({\Bbb S}[t_{2d}^{\pm 1}])_{> 0} \xrightarrow{} {\Bbb S}^{\rm triv} \oplus {\Bbb S}^{\rm triv}[1] \oplus {\rm THH}({\Bbb S}[t_{2d}^{\pm 1}])_{> 0} \oplus {\rm THH}({\Bbb S}[t_{2d}^{\pm 1}])_{< 0}\] of cyclotomic spectra. 
\end{proposition}

For ease of reference, we first record:

\begin{lemma}\label{lem:transretrspec} The left-hand vertical map \[{\rm THH}(p_*^{t_{2d} - {\rm nil}}) \colon {\rm THH}({\rm Perf}({\Bbb S})) \xrightarrow{} {\rm THH}({\rm Perf}({\Bbb S}[t_{2d}])^{t_{2d} - {\rm nil}})\] of \eqref{eq:t2dnilcofiberseq} admits a retraction in cyclotomic spectra. 
\end{lemma}

\begin{proof} By Schwede--Shipley \cite[Theorem 3.3]{SS03} and Antieau--Barthel--Gepner \cite[Theorem 1.11]{ABG18}, there is an equivalence \[{\rm Perf}({\Bbb S}[t_{2d}])^{t_{2d} - {\rm nil}} \simeq {\rm Perf}({\rm End}_{{\Bbb S}[t_{2d}]}({\Bbb S}))\] under which the transfer ${\rm Perf}({\Bbb S}) \to {\rm Perf}({\Bbb S}[t_{2d}])^{t_{2d} - {\rm nil}}$ is induced by the ${\Bbb E}_1$-ring map ${\Bbb S} \cong {\rm End}_{{\Bbb S}}({\Bbb S}) \to {\rm End}_{{\Bbb S}[t_{2d}]}({\Bbb S})$ given by restriction of scalars along ${\Bbb S}[t_{2d}] \xrightarrow{p} {\Bbb S}$. This ring map admits a retraction, by further restricting scalars along ${\Bbb S} \to {\Bbb S}[t_{2d}]$, giving a retraction ${\rm THH}({\Bbb S}) \xrightarrow{} {\rm THH}({\rm End}_{{\Bbb S}[t_{2d}]}({\Bbb S})) \xrightarrow{} {\rm THH}({\Bbb S})$ in cyclotomic spectra. 
\end{proof}

\begin{proof}[Proof of Proposition \ref{prop:controllingf}] By Lemma \ref{lem:transretrspec}, the map ${\rm THH}(p_*^{t_{2d} - {\rm nil}})$ splits off a cyclotomic summand in ${\rm THH}({\rm Perf}({\Bbb S}[t_{2d}])^{t_{2d} - {\rm nil}}) \simeq {\Bbb S}^{\rm triv} \oplus {\rm THH}({\Bbb S}[t_{2d}^{\pm 1}])_{< 0}[-1]$, cf.\  \eqref{eq:t2dnilsplit}. We also know that ${\rm THH}(p_*^{t_{2d} - {\rm nil}})$ must map null-homotopically to the negative weights ${\rm THH}({\Bbb S}[t_{2d}^{\pm 1}])_{< 0}[-1]$, since it refines to a map of graded cyclotomic spectra \eqref{eq:t2dnilcofiberseqgr}. The resulting equivalence of weight $0$ summands ${\Bbb S}^{\rm triv} \xrightarrow{\simeq} ({\Bbb S}^{\rm triv} \oplus {\rm THH}({\Bbb S}[t_{2d}^{\pm 1}])_{< 0}[-1])_0 = {\Bbb S}^{\rm triv}$ can be determined as the identity, for example by observing that ${\rm THH}(p_*^{t_{2d} - {\rm nil}})$ is induced by the map ${\Bbb S} \to {\rm End}_{{\Bbb S}[t_{2d}]}({\Bbb S})$ of graded ${\Bbb E}_1$-rings. Passing to cofibers, we obtain the desired description of $f$.  
\end{proof}

\subsection{Establishing the equivalence}\label{subsec:esteq} We are now in a position to prove:

\begin{proposition}\label{prop:eqgradedcyc} There is a canonical equivalence \[{\rm THH}^{\rm gr}({\Bbb S}[t_{2d}^{\pm 1}])_{\ge 0} \simeq {\rm THH}^{\rm gr}({\Bbb S}[t_{2d}] | t_{2d})\] of ${\rm THH}^{\rm gr}({\Bbb S}[t_{2d}])$-modules in graded cyclotomic spectra. 
\end{proposition}

\begin{proof} Let ${\rm Mod}_{{\rm THH}^{\rm gr}({\Bbb S}[t_{2d}])}({\rm GrCycSp})_{\ge 0} \subset {\rm Mod}_{{\rm THH}^{\rm gr}({\Bbb S}[t_{2d}])}({\rm GrCycSp})$ denote the full subcategory spanned by those objects whose underlying graded spectrum is concentrated in non-negative weights. The weight-connective cover ${\rm THH}^{\rm gr}({\Bbb S}[t_{2d}^{\pm 1}])_{\ge 0}$ belongs to this category by definition, while ${\rm THH}^{\rm gr}({\Bbb S}[t_{2d}] | t_{2d})$ belongs to it by construction: It is the cofiber of the map \eqref{eq:grtrans} between objects concentrated in non-negative weights. It thus suffices to show that the maps \[\begin{tikzcd}[row sep = small]{\rm Map}_{{\rm Mod}_{{\rm THH}^{\rm gr}({\Bbb S}[t_{2d}])}({\rm GrCycSp})}(X, {\rm THH}^{\rm gr}({\Bbb S}[t_{2d}^{\pm 1}])_{\ge 0}) \ar{d} \\ {\rm Map}_{{\rm Mod}_{{\rm THH}^{\rm gr}({\Bbb S}[t_{2d}])}({\rm GrCycSp})}(X, {\rm THH}^{\rm gr}({\Bbb S}[t_{2d}^{\pm 1}])) \\ {\rm Map}_{{\rm Mod}_{{\rm THH}^{\rm gr}({\Bbb S}[t_{2d}])}({\rm GrCycSp})}(X, {\rm THH}^{\rm gr}({\Bbb S}[t_{2d}] | t_{2d})) \ar{u}\end{tikzcd}\] induced by the inclusion of the non-negative weights and the map $f^{\rm gr}$ are equivalences for any $X \in {\rm Mod}_{{\rm THH}^{\rm gr}({\Bbb S}[t_{2d}])}({\rm GrCycSp})_{\ge 0}$. By Lemma \ref{lem:mappingspectramodules}, the above diagram may be rewritten as \[\begin{tikzcd}[row sep = small]{\lim}_{[n] \in \Delta}({\rm Map}_{{\rm GrCycSp}}(X \otimes {\rm THH}^{\rm gr}({\Bbb S}[t_{2d}])^{\otimes n}, {\rm THH}^{\rm gr}({\Bbb S}[t_{2d}^{\pm 1}])_{\ge 0})) \ar{d} \\ {\lim_{[n] \in \Delta}}({\rm Map}_{{\rm GrCycSp}}(X \otimes {\rm THH}^{\rm gr}({\Bbb S}[t_{2d}])^{\otimes n}, {\rm THH}^{\rm gr}({\Bbb S}[t_{2d}^{\pm 1}]))) \\ {\lim_{[n] \in \Delta}}({\rm Map}_{{\rm GrCycSp}}(X \otimes {\rm THH}^{\rm gr}({\Bbb S}[t_{2d}])^{\otimes n}, {\rm THH}^{\rm gr}({\Bbb S}[t_{2d}] | t_{2d}))) \ar{u}\end{tikzcd}\] involving limits of  mapping spaces in ${\rm GrCycSp}$. The domain in each mapping space remains concentrated in non-negative weights since ${\rm THH}^{\rm gr}({\Bbb S}[t_{2d}])$ is concentrated in non-negative weights. Hence the top arrow is an equivalence by construction, while the bottom is an equivalence by Proposition \ref{prop:controllingf}.
\end{proof}

\subsection{Identifying the residue maps} We now aim to identify the two ``residue'' maps \[{\rm THH}({\Bbb S}[t_{2d}^{\pm 1}])_{\ge 0} \xrightarrow{{\rm res}} {\Bbb S}^{\rm triv}[1] \text{ and } {\rm THH}({\Bbb S}[t_{2d}^{\pm 1}])_{\ge 0} \xrightarrow{\simeq} {\rm THH}({\Bbb S}[t_{2d}] | t_{2d}) \xrightarrow{\partial} {\Bbb S}^{\rm triv}[1]\] as maps of ${\rm THH}({\Bbb S}[t_{2d}])$-modules. Following a suggestion of Rognes, we do this by studying the mapping spectra encoding all possible residue maps. Explicit descriptions of the equivalent mapping spectra \[{\rm map}_{{\rm Mod}_{{\rm THH}({\Bbb S}[t_{2d}])}({\rm CycSp})}({\rm THH}({\Bbb S}[t_{2d}^{\pm 1}])_{\ge 0}, {\Bbb S}^{\rm triv}[1]) \text{ and}\] \[{\rm map}_{{\rm Mod}_{{\rm THH}({\Bbb S}[t_{2d}])}({\rm CycSp})}({\rm THH}({\Bbb S}[t_{2d}] | t_{2d}), {\Bbb S}^{\rm triv}[1])\] are probably complicated, and we see no reason for the forgetful functor to cyclotomic spectra to retain enough information to reduce the problem to the case of underlying cyclotomic spectra (which we solved in Corollary \ref{cor:residuecycl}).  

This is our main reason for working with \emph{graded} cyclotomic spectra. The forgetful functors \begin{equation}\label{eq:forgetfulfunctors}{\rm Mod}_{{\rm THH}^{\rm gr}({\Bbb S}[t_{2d}])}({\rm GrCycSp}) \to {\rm GrCycSp} \xrightarrow{{\rm und}} {\rm CycSp}\end{equation} involving \emph{graded} cyclotomic spectra capture all relevant information: 

\begin{theorem}\label{thm:gradedagreement} The map \begin{equation}\label{eq:gradedagreementdiagram}\begin{tikzcd}[row sep = small]{\rm map}_{{\rm Mod}_{{\rm THH}^{\rm gr}({\Bbb S}[t_{2d}])}({\rm GrCycSp})}({\rm THH}^{\rm gr}({\Bbb S}[t_{2d}^{\pm 1}])_{\ge 0}, {\Bbb S}^{\rm triv}[1]) \ar{d} \\ {\rm map}_{{\rm CycSp}}({\rm THH}({\Bbb S}[t_{2d}^{\pm 1}])_{\ge 0}, {\Bbb S}^{\rm triv}[1]) \end{tikzcd}\end{equation} induced by \eqref{eq:forgetfulfunctors} is equivalent to the inclusion \[{\rm TC}({\Bbb S}) \oplus {\rm TC}({\Bbb S})[1] \to {\rm TC}({\Bbb S}) \oplus {\rm TC}({\Bbb S})[1] \oplus {\rm map}_{{\rm CycSp}}({\rm THH}({\Bbb S}[t_{2d}^{\pm 1}])_{> 0}, {\Bbb S}^{\rm triv})[1]\] of spectra. The analogous statement holds for ${\rm THH}^{\rm gr}({\Bbb S}[t_{2d}] | t_{2d})$.  
\end{theorem}

The proof of Theorem \ref{thm:gradedagreement} requires some preparation. Throughout this, we always consider the sphere ${\Bbb S}$ with its trivial cyclotomic structure. We begin by identifying the source and target of \eqref{eq:gradedagreementdiagram}:  

\begin{proposition}\label{prop:equivalencetctc1spec} There is an equivalence of spectra \[{\rm map}_{{\rm Mod}_{{\rm THH}^{\rm gr}({\Bbb S}[t_{2d}])}({\rm GrCycSp})}({\rm THH}^{\rm gr}({\Bbb S}[t_{2d}^{\pm 1}])_{\ge 0}, {\Bbb S}[1]) \simeq {\rm TC}({\Bbb S}) \oplus {\rm TC}({\Bbb S})[1],\] and similarly for ${\rm THH}^{\rm gr}({\Bbb S}[t_{2d}] | t_{2d})$ by Proposition \ref{prop:eqgradedcyc}.  
\end{proposition}

\begin{proof} Using Lemma \ref{lem:gradedcontrol}, we find that forgetting the module structure induces an equivalence \[\begin{tikzcd}[row sep = small]{\rm map}_{{\rm Mod}_{{\rm THH}^{\rm gr}({\Bbb S}[t_{2d}])}({\rm GrCycSp})}({\rm THH}^{\rm gr}({\Bbb S}[t_{2d}^{\pm 1}])_{\ge 0}, {\Bbb S}[1]) \ar{d}{\simeq} \\ {\rm map}_{{{\rm GrCycSp}}}({\rm THH}^{\rm gr}({\Bbb S}[t_{2d}^{\pm 1}])_{\ge 0}, {\Bbb S}[1])\end{tikzcd}\] of mapping spectra. Since ${\Bbb S}[1]$ is concentrated in weight zero, the inclusion of the weight zero component ${\Bbb S} \oplus {\Bbb S}[1] \to {\rm THH}^{\rm gr}({\Bbb S}[t_{2d}^{\pm 1}])_{\ge 0}$ induces an equivalence \[{\rm map}_{{{\rm GrCycSp}}}({\rm THH}^{\rm gr}({\Bbb S}[t_{2d}^{\pm 1}])_{\ge 0}, {\Bbb S}[1]) \xrightarrow{\simeq} {\rm map}_{{{\rm CycSp}}}({\Bbb S} \oplus {\Bbb S}[1], {\Bbb S}[1]).\] The result now follows from the definition of ${\rm TC}$ \cite[Definition II.1.8(i)]{NS18}. 
\end{proof}

\begin{lemma}\label{lem:equivalencetctc1tr1spec} The spectrum ${\rm map}_{\rm CycSp}({\rm THH}({\Bbb S}[t_{2d}^{\pm 1}])_{\ge 0}, {\Bbb S}[1])$ is equivalent to \[{\rm TC}({\Bbb S}) \oplus {\rm TC}({\Bbb S})[1] \oplus {\rm map}_{{\rm CycSp}}({\rm THH}({\Bbb S}[t_{2d}^{\pm 1}])_{> 0}, {\Bbb S}[1]),\] and similarly for ${\rm THH}({\Bbb S}[t_{2d}] | t_{2d})$ by Proposition \ref{prop:eqgradedcyc}. 
\end{lemma}

\begin{proof} We use the decomposition \eqref{eq:diagramofgradedcycsplit} identifying ${\rm THH}({\Bbb S}[t_{2d}^{\pm 1}])_{\ge 0}$ with  \[{\Bbb S} \oplus {\Bbb S}[1] \oplus {\rm THH}({\Bbb S}[t_{2d}^{\pm 1}])_{> 0}\] and conclude by the definition of ${\rm TC}$. 
\end{proof}

\begin{remark} For a cyclotomic spectrum $X$, the mapping spectrum \[{\rm map}_{{\rm CycSp}}({\rm THH}({\Bbb S}[t_{2d}^{\pm 1}])_{> 0}, X)\] is a variant of ${\rm TR}$ of $X$, where the role of the flat affine line ${\Bbb S}[t]$ is replaced by its graded analog ${\Bbb S}[t_{2d}]$. Using the setup of Example \ref{ex:degreezero}, we find that the spectrum of Lemma \ref{lem:equivalencetctc1tr1spec} admits the clean description ${\rm TC}({\Bbb S}) \oplus {\rm TC}({\Bbb S})[1] \oplus {\rm TR}({\Bbb S})[1]$ by corepresentability of ${\rm TR}$ by ${\Bbb S}[B^{\rm cyc}_{> 0}(\langle t \rangle)]$ (cf.\ \cite[Theorem B]{McC24}) when $d = 0$. 
\end{remark}

\begin{proof}[Proof of Theorem \ref{thm:gradedagreement}] By Proposition \ref{prop:equivalencetctc1spec} and Lemma \ref{lem:equivalencetctc1tr1spec}, \eqref{eq:gradedagreementdiagram} is indeed a map of the form \[{\rm TC}({\Bbb S}) \oplus {\rm TC}({\Bbb S})[1] \to {\rm TC}({\Bbb S}) \oplus {\rm TC}({\Bbb S})[1] \oplus {\rm map}_{{\rm CycSp}}({\rm THH}({\Bbb S}[t_{2d}^{\pm 1}])_{> 0}, {\Bbb S}[1]).\] It remains to see that it is homotopic to the inclusion. For this, we consider the diagram \[\begin{tikzpicture}[baseline= (a).base]
\node[scale=.9] (a) at (0,0){\begin{tikzcd}[row sep = small, column sep = tiny]{\rm map}_{{\rm Mod}_{{\rm THH}^{\rm gr}({\Bbb S}[t_{2d}])}({\rm GrCycSp})}({\rm THH}^{\rm gr}({\Bbb S}[t_{2d}^{\pm 1}])_{\ge 0}, {\Bbb S}^{\rm triv}[1]) \ar[swap]{d}{\simeq} \ar{d}{\text{Lemma \ref{lem:gradedcontrol}}} \\ {\rm map}_{{\rm GrCycSp}}({\rm THH}^{\rm gr}({\Bbb S}[t_{2d}^{\pm 1}])_{\ge 0}, {\Bbb S}^{\rm triv}[1]) \ar{d} \ar{r}{\simeq} & {\rm map}_{{\rm CycSp}}({\Bbb S}^{\rm triv} \oplus {\Bbb S}^{\rm triv}[1], {\Bbb S}^{\rm triv}[1]) \ar[bend left = 3 mm]{dl} \\ {\rm map}_{{\rm CycSp}}({\rm THH}({\Bbb S}[t_{2d}^{\pm 1}])_{\ge 0}, {\Bbb S}^{\rm triv}[1]) \end{tikzcd}};\end{tikzpicture}\] of mapping spectra. The horizontal equivalence, induced by the inclusion of the weight zero component ${\Bbb S}^{\rm triv} \oplus {\Bbb S}^{\rm triv}[1] \xrightarrow{} {\rm THH}^{\rm gr}({\Bbb S}[t_{2d}^{\pm 1}])_{\ge 0}$, comes to life from ${\Bbb S}^{\rm triv}[1]$ being concentrated in weight zero. The bent arrow is induced by the projection ${\rm THH}({\Bbb S}[t_{2d}^{\pm 1}])_{\ge 0}\to {\Bbb S}^{\rm triv} \oplus {\Bbb S}^{\rm triv}[1]$, which gives precisely the desired inclusion. 
\end{proof}

\begin{corollary}\label{cor:residuegradedspec} The residue maps \[{\rm THH}^{\rm gr}({\Bbb S}[t_{2d}^{\pm 1}])_{\ge 0} \xrightarrow{{\rm res}} {\Bbb S}^{\rm triv}[1] \quad \!\!\!\!\! \text{and}  \quad \!\!\!\!{\rm THH}^{\rm gr}({\Bbb S}[t_{2d}^{\pm 1}])_{\ge 0} \xrightarrow{\simeq} {\rm THH}^{\rm gr}({\Bbb S}[t_{2d}] | t_{2d}) \xrightarrow{\partial} {\Bbb S}^{\rm triv}[1]\] are homotopic as maps of ${\rm THH}^{\rm gr}({\Bbb S}[t_{2d}])$-modules in graded cyclotomic spectra. 
\end{corollary}

\begin{proof} They are homotopic as maps of cyclotomic spectra by Corollary \ref{cor:residuecycl}. But in the diagram \[\begin{tikzpicture}[baseline= (a).base]
\node[scale=.73] (a) at (0,0){\begin{tikzcd}[column sep = small]{\rm map}_{{\rm Mod}_{{\rm THH}^{\rm gr}({\Bbb S}[t_{2d}])}({\rm GrCycSp})}({\rm THH}^{\rm gr}({\Bbb S}[t_{2d}^{\pm 1}])_{\ge 0}, {\Bbb S}^{\rm triv}[1]) \ar[swap]{d}{\simeq} \ar{d}{\text{Lemma \ref{lem:gradedcontrol}}} \ar{r}{\simeq} & {\rm map}_{{\rm Mod}_{{\rm THH}^{\rm gr}({\Bbb S}[t_{2d}])}({\rm GrCycSp})}({\rm THH}^{\rm gr}({\Bbb S}[t_{2d}] | t_{2d}), {\Bbb S}^{\rm triv}[1]) \ar[swap]{d}{\simeq} \ar{d}{\text{Lemma \ref{lem:gradedcontrol}}}   \\ {\rm map}_{{\rm GrCycSp}}({\rm THH}^{\rm gr}({\Bbb S}[t_{2d}^{\pm 1}])_{\ge 0}, {\Bbb S}^{\rm triv}[1]) \ar{d} & {\rm map}_{{\rm GrCycSp}}({\rm THH}^{\rm gr}({\Bbb S}[t_{2d}] | t_{2d}), {\Bbb S}^{\rm triv}[1]) \ar{d} \\ {\rm map}_{{\rm CycSp}}({\rm THH}({\Bbb S}[t_{2d}^{\pm 1}])_{\ge 0}, {\Bbb S}^{\rm triv}[1]) \ar{r}{\simeq} &  {\rm map}_{{\rm CycSp}}({\rm THH}({\Bbb S}[t_{2d}] | t_{2d}), {\Bbb S}^{\rm triv}[1]) \end{tikzcd}};\end{tikzpicture}\] with vertical maps induced by \eqref{eq:forgetfulfunctors} and horizontal equivalences induced by Proposition \ref{prop:eqgradedcyc} and Corollary \ref{cor:undcycspspec}, both vertical compositions are inclusions by Theorem \ref{thm:gradedagreement}. Hence the two residue maps are homotopic in ${\rm Mod}_{{\rm THH}^{\rm gr}({\Bbb S}[t_{2d}])}({\rm GrCycSp})$, since they are homotopic in ${\rm CycSp}$. 
\end{proof}

\subsection{Finalizing the proof of Theorem \ref{thm:mainthmspec}}\label{subsec:finalizingproof} Base-changing the equivalence of Proposition \ref{prop:eqgradedcyc} of underlying ${\rm THH}({\Bbb S}[t_{2d}])$-modules in cyclotomic spectra along the map ${\rm THH}({\Bbb S}[t_{2d}]) \xrightarrow{t_{2d} \mapsto x} {\rm THH}(A)$ to obtain an equivalence \begin{equation}\label{eq:compmapalmosttherespec}{\rm THH}(A) \otimes_{{\rm THH}({\Bbb S}[t_{2d}])} {\rm THH}({\Bbb S}[t_{2d}^{\pm 1}])_{\ge 0} \xrightarrow{} {\rm THH}(A) \otimes_{{\rm THH}({\Bbb S}[t_{2d}])} {\rm THH}({\Bbb S}[t_{2d}] | t_{2d})\end{equation} of ${\rm THH}(A)$-modules in cyclotomic spectra. By definition, the source of this map is ${\rm THH}(A, \langle x \rangle)$. We now identify the target:

\begin{lemma}\label{lem:transbasechangespec} There is an equivalence \[{\rm THH}(A) \otimes_{{\rm THH}({\Bbb S}[t_{2d}])} {\rm THH}({\Bbb S}[t_{2d}] | t_{2d}) \xrightarrow{} {\rm THH}(A | x)\] of ${\rm THH}(A)$-modules in cyclotomic spectra. 
\end{lemma}

\begin{proof} We first observe that the diagram \[\begin{tikzcd}{\rm Perf}({\Bbb S}) \ar{r}{p_*} \ar{d} & {\rm Perf}({\Bbb S}[t_{2d}]) \ar{d} \\ {\rm Perf}(A / x) \ar{r}{p_*} & {\rm Perf}(A)\end{tikzcd}\] commutes, with vertical arrows defined by base-change along the maps  ${\Bbb S} \to  A / x$ and ${\Bbb S}[t_{2d}] \xrightarrow{t_{2d} \mapsto x} A$, respectively. This gives rise to the commutative diagram \[\begin{tikzcd}{\rm THH}({\rm Perf}({\Bbb S})) \ar{r}{{\rm THH}(p_*)} \ar{d} & {\rm THH}({\rm Perf}({\Bbb S}[t_{2d}])) \ar{d} \\ {\rm THH}({\rm Perf}(A / x)) \ar{r}{{\rm THH}(p_*)} & {\rm THH}({\rm Perf}(A))\end{tikzcd}\] of ${\rm THH}({\rm Perf}({\Bbb S}[t_{2d}]))$-modules by Lemma \ref{lem:projectionformulagr}. It thus base-changes to a diagram \[\begin{tikzpicture}[baseline= (a).base]
\node[scale=.74] (a) at (0,0){\begin{tikzcd}{\rm THH}({\rm Perf}(A)) \otimes_{{\rm THH}({\rm Perf}({\Bbb S}[t_{2d}]))} {\rm THH}({\rm Perf}({\Bbb S})) \ar{d}{\simeq} \ar{rr}{{\rm id} \otimes {\rm THH}(p_*)} & & {\rm THH}({\rm Perf}(A)) \otimes_{{\rm THH}({\rm Perf}({\Bbb S}[t_{2d}]))} {\rm THH}({\rm Perf}({\Bbb S}[t_{2d}])) \ar{d}{\simeq} \\ {\rm THH}({\rm Perf}(A / x)) \ar{rr}{{\rm THH}(p_*)} & & {\rm THH}({\rm Perf}(A))\end{tikzcd}};\end{tikzpicture}\] of ${\rm THH}({\rm Perf}(A))$-modules, where the left-hand vertical equivalence is Morita invariance and symmetric monoidality of ${\rm THH}$. The desired equivalence is that of horizontal cofibers. 
\end{proof}

\begin{proof}[Proof of Theorem \ref{thm:mainthmspec}] Both \eqref{eq:compmapalmosttherespec} and the map of Lemma \ref{lem:transbasechangespec} are equivalences. The assertion about the residue maps follows from Corollary \ref{cor:residuegradedspec}. 
\end{proof}

Theorem \ref{thm:mainthm} may now be obtained as a special case with $A$ discrete and $d = 0$, or by copying the above argument in the context of Example \ref{ex:degreezero}. In particular, the structure map ${\Bbb S}[t] \xrightarrow{t \mapsto x} A$ and ${\rm THH}(A, \langle x \rangle)$ admit ${\Bbb E}_{\infty}$-structures in this case. 

We end this section with the following two remarks:

\begin{remark}\label{rem:rognesconj} Let us explain how to recover the exact formulation of \cite[Conjecture 7.5]{Rog14} as a consequence of Theorem \ref{thm:mainthmspec}. We recall its context: $E$ is an even-periodic ${\Bbb E}_{\infty}$-ring of period $2d$, $e = \tau_{\ge 0}(E)$ is its connective cover. We ask that $\pi_0(e)$ is regular, and moreover that $\tau_{\le (2d - 1)}(e) = \pi_0(e)$. The first and last hypotheses imply that $e$ is even with homotopy groups concentrated in degrees multiples of $2d$ so that the setup of \cite{RSS26} allows for the construction of ${\rm THH}(e, \langle x \rangle)$ with $x \in \pi_{2d}(e)$ the periodicity class.  

The conjecture predicts that \begin{equation}\label{eq:rognesconj}{\rm THH}(e, \langle x \rangle) \simeq {\rm WTHH}^\Gamma(e | E)\end{equation} as cyclotomic spectra, where ${\rm WTHH}^\Gamma(e | E)$ is the Blumberg--Mandell term \cite{BM20} participating in the localization sequence \[{\rm THH}(\pi_0(e)) \to {\rm THH}(e) \to {\rm WTHH}^\Gamma(e | E),\] and so it is equivalent to what we have written as ${\rm THH}(e | x)$. The left-hand side ${\rm THH}(e, \langle x \rangle)$ is strictly speaking described as the ${\rm THH}$ of the direct image log structure $j_*{\rm GL}_1^{\cal J}(E)$ formed using the model of \cites{RSS15, RSS18}: Under the present hypotheses, \cite[Lemma B.5]{RSS26} provides an equivalence ${\rm THH}(e, \langle x \rangle) \simeq {\rm THH}(e, j_*{\rm GL}_1^{\cal J}(E))$. As no cyclotomic structure has been constructed on the model ${\rm THH}(e, j_*{\rm GL}_1^{\cal J}(E))$, we interpret the conjecture as referring to the cyclotomic structure on the equivalent term ${\rm THH}(e, \langle x \rangle)$.  

Altogether, Theorem \ref{thm:mainthmspec} produces an equivalence \eqref{eq:rognesconj} of cyclotomic spectra with the predicted identification of the boundary maps. 
\end{remark}

\begin{remark}\label{rem:koinclude} Most of our arguments may be adapted to the context of \cites{RSS15, RSS18}. While the setup of \cite{RSS26} does not include a pre-log structure of the form $({\rm ko}, \langle w \rangle)$ generated by the Bott class $w \in \pi_8({\rm ko})$, this works in their old model. If one writes down a cyclotomic structure on ${\rm THH}({\Bbb S}^{\cal J}[D(w)], D(w))$ (which plays the role of ${\rm THH}({\Bbb S}[t_{2d}], \langle t_{2d} \rangle)$ in this setting) with the expected properties, we obtain compatible horizontal cofiber sequences \[\begin{tikzcd}K(\tau_{\le 7}({\rm ko})) \ar{r} \ar{d} & K({\rm ko}) \ar{r} \ar{d} & K({\rm Perf}({\rm ko} | w)) \ar{d} \\ {\rm TC}(\tau_{\le 7}({\rm ko})) \ar{r} & {\rm TC}({\rm ko}) \ar{r} & {\rm TC}({\rm ko}, \langle w\rangle)\end{tikzcd}\] that differ from the cofiber sequences \[\begin{tikzcd}K({\Bbb Z}) \ar{r} \ar{d} & K({\rm ko}) \ar{r} \ar{d} & K({\rm KO}) \ar{d} \\ {\rm TC}({\Bbb Z}) \ar{r} & {\rm TC}({\rm ko}) \ar{r} & {\rm TC}({\rm ko} |  {\rm KO})\end{tikzcd}\] obtained from combining Barwick--Lawson \cite{BL14} and Blumberg--Mandell \cite{BM20}. 
\end{remark}

\section{Compatibility with the Hesselholt--Madsen construction}\label{sec:hmcomp} We now proceed to work towards the proof of Theorem \ref{thm:hmcomp}, making active use of the constructions of the previous sections. To begin, let us first recall the general notion of a log ring and how topological Hochschild homology is defined in this level of generality: 

\subsection{Pre-log and log rings} A \emph{pre-log ring} $(A, M) = (A, M, \alpha)$ consists of a commutative ring $A$, a commutative monoid $M$, and a map $\alpha \colon M \to (A, \cdot)$ of commutative monoids, adjoint to a map $\bar{\alpha} \colon {\Bbb Z}[M] \to A$ of commutative rings. It is a \emph{log ring} if the canonical map $\alpha^{-1}{\rm GL}_1(A) \xrightarrow{} {\rm GL}_1(A)$ is an isomorphism. The category of log rings embeds fully faithfully into that of pre-log rings, and the inclusion admits a left adjoint $(A, M, \alpha) \mapsto (A, M^a, \alpha^a)$ determined by \[M^a := M \oplus_{\alpha^{-1}{\rm GL}_1(A)} {\rm GL}_1(A) \xrightarrow{\alpha^a} (A, \cdot),\] where $\alpha^a$ is defined by $\alpha$ and the inclusion of units. 

\subsection{The log differentials} Let $(A, M)$ be a pre-log ring. The \emph{log differentials} $\Omega^1_{(A, M)}$ can be defined as the (ordinary, non-derived) pushout of the diagram \[\Omega^1_{A} \xleftarrow{ad\alpha(m)\mapsfrom a \otimes d(m)} A \otimes_{{\Bbb Z}[M]} \Omega^1_{{\Bbb Z}[M]} \xrightarrow{a \otimes d(m) \mapsto a\alpha(m) \otimes \gamma(m)} A \otimes_{{\Bbb Z}} M^{\rm gp}\] of $A$-modules, cf.\ \cite[(1.7)]{Kat89}.

\subsection{Logarithmic topological Hochschild homology}\label{subsec:logthh} Recall the replete bar construction $B^{\rm rep}(M)$ of Section \ref{subsec:repbar}. The following definition is due to Rognes: 

\begin{definition}[{\cite[Definition 8.11]{Rog09}}]\label{def:logthh} Let $(A, M)$ be a pre-log ring. The \emph{log topological Hochschild homology} ${\rm THH}(A, M)$ is the ${\Bbb E}_{\infty}$-ring given by the relative tensor product \[{\rm THH}(A, M) := {\rm THH}(A) \otimes_{{\Bbb S}[B^{\rm cyc}(M)]} {\Bbb S}[B^{\rm rep}(M)]\] of ${\Bbb E}_{\infty}$-rings.
\end{definition}

There are isomorphisms \[\pi_0{\rm THH}(A, M) \cong A \quad \text{and} \quad \pi_1{\rm THH}(A, M) \cong \Omega^1_{(A, M)}\] obtained by combining \cite[Proposition 5.13]{BLPO23} with \cite[Corollary 3.4, Lemma 3.5]{BLPO23Prism}, cf.\ \cite[Definition 5.27]{Rog09} and the ensuing discussion. By \cite[Construction 3.7]{BLPO23Prism}, we find that ${\rm THH}(A, M)$ is naturally an ${\Bbb E}_{\infty}$-algebra in cyclotomic spectra. We set ${\rm TC}(A, M) := {\rm TC}({\rm THH}(A, M))$. 

\begin{remark}\label{rem:logthhinv} By \cite[Theorem 4.24]{RSS15}, logarithmic ${\rm THH}$ is invariant under the logification construction: The canonical map ${\rm THH}(A, M) \xrightarrow{} {\rm THH}(A, M^a)$ is an equivalence. In particular, if ${\cal O}_K$ is a discrete valuation ring with chosen uniformizer $\pi$, the canonical map ${\rm THH}({\cal O}_K, \langle \pi \rangle) \xrightarrow{} {\rm THH}({\cal O}_K, {\cal O}_K \cap {\rm GL}_1(K))$ is an equivalence. In this way, logarithmic ${\rm THH}$ is independent of the choice of uniformizer. Analogously, a variant of logification invariance is established for general log ring spectra in \cite[Theorem 8.15]{RSS26}, rendering the construction of ${\rm THH}(A, \langle x \rangle)$ studied in Section \ref{sec:logthhloc} independent of the choice of map ${\Bbb S}[t_{2d}] \xrightarrow{} A$ capturing $x \in \pi_{2d}(A)$, at least in the examples discussed in \cite[Example 8.13]{RSS26}.
\end{remark}

\subsection{The log differential graded ring $\pi_*{\rm THH}(A, M)$}\label{subsec:logthhlogdiff} In general, the term ${\rm THH}(A, \langle x \rangle)$ carries \emph{a priori} more structure than the abstractly defined cofiber ${\rm THH}(A | x)$; it is an ${\Bbb E}_{\infty}$-ring, whose coefficient ring forms a log differential graded ring in the following sense (cf.\ \cite{Kat89} or \cite[p.\ 4]{HM03}):

\begin{definition}\label{def:logdiff} A \emph{log differential graded ring} $(B_*, M) = (B_*, M, d, \alpha, {\rm dlog})$ is
\begin{enumerate}
\item[(a)] a differential graded ring $(B_*, d)$;
\item[(b)] a pre-log ring $(B_0, M) = (B_0, M, \alpha)$; and
\item[(c)] a map ${\rm dlog} \colon M \to (B_1, +)$ of commutative monoids
\end{enumerate}
satisfying the relations $d(\alpha(m)) = \alpha(m){\rm dlog}(m)$ and $d({\rm dlog}(m)) = 0$. 
\end{definition}

We now explain how $\pi_*{\rm THH}(A, M)$ is an example for general $(A, M)$. The $S^1$-action on the spectrum ${\rm THH}(A, M)$ induces a map \[\delta \colon \pi_*{\rm THH}(A, M) \to \pi_{* + 1}{\rm THH}(A, M).\] We have $\delta^2 = \eta\delta = 0$, since ${\rm THH}(A, M)$ is an $A$-algebra.  The map $\delta$ exhibits $\pi_*{\rm THH}(A, M)$ as a differential graded ring (cf.\ \cite[Section 1]{Hes96}). In degree zero, we choose the given pre-log structure $(A, M)$. We define \[{\rm dlog} \colon M \to \pi_1{\rm THH}(A, M) \cong \Omega^1_{(A, M)}\] to be the canonical map $m \mapsto {\rm dlog}(m)$. This is indeed a log differential graded ring:

\begin{proposition}\label{prop:logthhlogdiff} The structure maps described above exhibit $\pi_*{\rm THH}(A, M)$ as a log differential graded ring. 
\end{proposition}

\begin{proof} The proof of \cite[Proposition 1.4.5]{Hes96} shows that the linearization map ${\rm THH}(A, M) \to {\rm HH}(A, M)$ participates in the left-hand commutative diagram \[\begin{tikzcd}\pi_n{\rm THH}(A, M) \ar{d}{\delta} \ar{r} & \pi_n{\rm HH}(A, M) \ar{d}{B} & \Omega^n_{(A, M)} \ar{d}{d} \ar{l} \\ \pi_{n + 1}{\rm THH}(A, M) \ar{r} & \pi_{n + 1}{\rm HH}(A, M) & \Omega^{n + 1}_{(A, M)} \ar{l},\end{tikzcd}\] where $B$ is Connes' operator. The right-hand commutative diagram is \cite[Proposition 2.12]{BLPO23Prism}. From this, the relation $\delta(\alpha(m)) = \alpha(m){\rm dlog}(m)$ follows from all horizontal maps being isomorphisms in degrees $0$ and $1$ and the relation $d(\alpha(m)) = \alpha(m){\rm dlog}(m)$ in $\Omega^1_{(A, M)}$. 

It remains to prove that $\delta({\rm dlog}(m)) = 0$. As the linearization map is still an isomorphism in degree $2$, this follows from the relation $B({\rm dlog}(m)) = 0$ established in the proof of \cite[Proposition 2.12]{BLPO23Prism}. 
\end{proof}

\subsection{The special case of regular quotients} We now consider the case of pre-log rings of the form $(A, \langle x \rangle)$ with $A$ regular Noetherian, and we moreover assume that the quotient $A/x$ is regular. Then d\'evissage implies that there is a diagram \[\begin{tikzcd}K(A/x) \ar{r}{K(p_*)} \ar{d} & K(A) \ar{d} \ar{r} & K(A[x^{-1}]) \ar{d} \\ {\rm THH}(A/x) \ar{r}{{\rm THH}(p_*)} & {\rm THH}(A) \ar{r} & {\rm THH}(A | x)\end{tikzcd}\] of horizontal cofiber sequences. We have an explicit model of the right-hand trace map in mind, namely \[K(A[x^{-1}]) \xleftarrow{\simeq} K({\rm Perf}(A | x)) \xrightarrow{{\rm tr}} {\rm THH}({\rm Perf}(A | x)) =: {\rm THH}(A | x),\]  where ${\rm Perf}(A | x)$ is defined as the cofiber in ${\rm Cat}_{\infty}^{\rm perf}$ of the fully faithful functor appearing in the functorial factorization of the transfer $p_* \colon {\rm Perf}(A/x) \to {\rm Perf}(A)$ provided by \cite[Proposition 2.8]{RSW25}. With our hypotheses, the left-hand map $K({\rm Perf}(A | x)) \xrightarrow{} K(A[x^{-1}])$ is an equivalence by localization. 

The following statement makes reference to the map \[{\rm dlog} \colon {\Bbb S}[B\langle t \rangle^{\rm gp}] \xrightarrow{} {\Bbb S}[t] \otimes {\Bbb S}[B\langle t \rangle^{\rm gp}] \cong {\Bbb S}[B^{\rm rep}(\langle t \rangle)] \cong {\rm THH}({\Bbb S}[t], \langle t \rangle)\] of ${\Bbb E}_{\infty}$-rings. 

\begin{proposition}\label{prop:dlogcomp} The diagram \[\begin{tikzcd}{\Bbb S}[B\langle t \rangle^{\rm gp}] \ar{r}{t \mapsto x} \ar{d}{{\rm dlog}} & {\Bbb S}[{\rm BGL}_1(A[x^{-1}])] \ar{r} & K(A[x^{-1}]) \ar{d}{{\rm tr}} \\ {\rm THH}({\Bbb S}[t], \langle t \rangle) \ar{r}{\simeq} & {\rm THH}({\Bbb S}[t] | t) \ar{r}{t \mapsto x} & {\rm THH}(A | x)\end{tikzcd}\] of spectra commutes. 
\end{proposition}

\begin{proof} The diagram in question is equivalent to the outer diagram in \begin{equation}\label{eq:dlogcompproofdiagram}\begin{tikzcd}{\Bbb S}[B\langle t \rangle^{\rm gp}] \ar{r} \ar{d}{{\rm dlog}} & {\Bbb S}[{\rm BGL}_1(A[t^{\pm 1}])] \ar{r} & K(A[t^{\pm 1}]) \ar{d}{{\rm tr}}\ar{r}{t \mapsto x} & K(A[x^{\pm 1}]) \ar{d}{{\rm tr}}\\ {\rm THH}({\Bbb S}[t], \langle t \rangle) \ar{r}{\simeq} & {\rm THH}({\Bbb S}[t] | t) \ar{r} & {\rm THH}(A[t] | t) \ar{r}{t \mapsto x} & {\rm THH}(A | x),\end{tikzcd}\end{equation} and so it suffices to prove that the left-hand square commutes. The composition involving ${\rm dlog}$ is equivalent to the composition \begin{equation}\label{eq:dlogcompproofdiagram2}\begin{tikzcd}[row sep = small]{\Bbb S}[B\langle t \rangle^{\rm gp}] \ar{r} & {\rm THH}(A) \otimes {\Bbb S}[B\langle t \rangle^{\rm gp}] \ar{r}{{\rm id} \otimes {\rm dlog}} &  {\rm THH}(A) \otimes {\rm THH}({\Bbb S}[t], \langle t \rangle) \ar{d}{\simeq} \\ \vspace{10 mm} & &  {\rm THH}(A) \otimes {\rm THH}({\Bbb S}[t] | t) \ar{d}{\cong} \\ \vspace{10 mm} & & {\rm THH}(A[t]) \otimes_{{\rm THH}({\Bbb S}[t])} {\rm THH}({\Bbb S}[t] | t) \ar{d}{\simeq} \ar[swap]{d}{\text{Lemma \ref{lem:transbasechangespec}}} \\ \vspace{10 mm} & &  {\rm THH}(A[t] | t).\end{tikzcd}\end{equation} In other words, it is equivalent to the map ${\Bbb S}[B\langle t \rangle^{\rm gp}] \to {\rm THH}(A) \otimes {\Bbb S}[B\langle t \rangle^{\rm gp}]$ induced by the unit of ${\rm THH}(A)$ followed by the inclusion of summands \[{\rm THH}(A) \oplus {\rm THH}(A)[1] \to  {\rm THH}(A) \oplus {\rm THH}(A)[1] \oplus {\rm THH}(A)[B^{\rm cyc}_{> 0}(\langle t \rangle^{\rm gp})],\] cf.\ \eqref{eq:susprepsplit}. We now wish to make the same statement about the composition in the left-hand diagram of \eqref{eq:dlogcompproofdiagram} involving the trace. For this, we postcompose this composition with ${\rm THH}(A[t] | t) \to {\rm THH}(A[t^{\pm 1}])$; that is, we consider the composition \[{\Bbb S}[B\langle t \rangle^{\rm gp}] \xrightarrow{} {{\Bbb S}[{\rm BGL}_1(A[t^{\pm 1}])]} \xrightarrow{} K(A[t^{\pm 1}]) \xrightarrow{{\rm tr}} {\rm THH}(A[t] | t) \xrightarrow{} {\rm THH}(A[t^{\pm 1}]),\] where the composition of the last two maps is by construction the usual trace $K(A[t^{\pm 1}]) \to {\rm THH}(A[t^{\pm 1}])$. By Antieau--Barthel--Gepner \cite[Proposition 2.5]{ABG18}, this is equivalent to the upper horizontal composition in the commutative diagram \[\begin{tikzpicture}[baseline= (a).base]
\node[scale=.88] (a) at (0,0){\begin{tikzcd}[column sep = small, row sep = small]{\Bbb S}[B\langle t \rangle^{\rm gp}]  \ar[swap,bend right = 3 mm]{ddrr}{\eqref{eq:dlogcompproofdiagram2}}\ar{r}{{\rm dlog}} & {\rm THH}({\Bbb S}[t], \langle t \rangle) \ar{r}{\simeq} & {\rm THH}({\Bbb S}[t] | t) \ar{r} \ar{d} &  {\rm THH}({\Bbb S}[t^{\pm 1}]) \ar{r} & {\rm THH}(A[t^{\pm 1}]). \\ \vspace{10 mm} & & {\rm THH}(A) \otimes {\rm THH}({\Bbb S}[t] | t) \ar{d}{\simeq} \\ \vspace{10 mm} & & {\rm THH}(A[t] | t) \ar[bend right = 3 mm]{uurr}\end{tikzcd}};\end{tikzpicture}\] Since the map ${\rm THH}(A[t] | t) \to {\rm THH}(A[t^{\pm 1}])$ is equivalent to the inclusion \[\begin{tikzcd}[row sep = small]{\rm THH}(A) \oplus {\rm THH}(A)[1] \oplus {\rm THH}(A)[B^{\rm cyc}_{> 0}(\langle t \rangle^{\rm gp})] \ar{d} \\ {\rm THH}(A) \oplus {\rm THH}(A)[1] \oplus {\rm THH}(A)[B^{\rm cyc}_{> 0}(\langle t \rangle^{\rm gp})] \oplus {\rm THH}(A)[B^{\rm cyc}_{< 0}(\langle t \rangle^{\rm gp})] \end{tikzcd}\] by Proposition \ref{prop:controllingf} and Lemma \ref{lem:transbasechangespec}, we find that the two compositions already agree in ${\rm THH}(A[t] | t)$ (as both are ${\Bbb S}[B\langle t \rangle^{\rm gp}] \to {\rm THH}(A) \otimes {\Bbb S}[B\langle t \rangle^{\rm gp}]$ followed by the inclusion), which concludes the proof. 
\end{proof}

\subsection{Recollections from \cite{HM03}}\label{subsec:hmrecall} Suppose now that ${\cal O}_K$ is a complete discrete valuation ring of mixed characteristic with perfect residue field $k = {\cal O}_K / \pi$ and fraction field $K = {\cal O}_K[\pi^{-1}]$. We will now recall the (output of the) construction of the term ${\rm THH}({\cal O}_K | K)$. 

For a commutative ring $A$, let ${\cal P}(A)$ denote the category of perfect complexes of $A$-modules. For $A = {\cal O}_K$, there are two natural choices of weak equivalences on ${\cal P}({\cal O}_K)$ that participate in the structure of a Waldhausen category: the quasi-isomorphisms $w{\cal P}({\cal O}_K)$, and the maps that become quasi-isomorphisms after base-change along ${\cal O}_K \to K$, denoted $v{\cal P}({\cal O}_K)$. Waldhausen's fibration theorem \cite[Theorem 1.6.4]{Wal85} gives the upper horizontal cofiber sequence in the diagram \begin{equation}\label{eq:waldcof}\begin{tikzpicture}[baseline= (a).base]
\node[scale=.94] (a) at (0,0){\begin{tikzcd}[column sep = tiny] K({\cal P}({\cal O}_K)^{v{\cal P}({\cal O}_K)}, w{\cal P}({\cal O}_K)) \ar{r} \ar{d}{{\rm tr}} & K({\cal P}({\cal O}_K), w{\cal P}({\cal O}_K)) \ar{r} \ar{d}{{\rm tr}} & K({\cal P}({\cal O}_K), v{\cal P}({\cal O}_K)) \ar{d}{{\rm tr}} \\ {\rm THH}({\cal P}({\cal O}_K)^{v{\cal P}({\cal O}_K)}, w{\cal P}({\cal O}_K)) \ar{r} & {\rm THH}({\cal P}({\cal O}_K), w{\cal P}({\cal O}_K)) \ar{r} & {\rm THH}({\cal P}({\cal O}_K), v{\cal P}({\cal O}_K)),\end{tikzcd}};\end{tikzpicture}\end{equation} where ${\cal P}({\cal O}_K)^{v{\cal P}({\cal O}_K)}$ is the full subcategory of ${\cal P}({\cal O}_K)$ consisting of those complexes that are $v{\cal P}({\cal O}_K)$-acyclic; that is, those that become acyclic after base-change along ${\cal O}_K \to K$. Waldhausen's approximation theorem \cite[Theorem 1.6.7]{Wal85} implies that extension of scalars induces an equivalence $K({\cal P}({\cal O}_K), v{\cal P}({\cal O}_K)) \xrightarrow{} K({\cal P}(K), w{\cal P}(K)) \simeq K(K)$, while d\'evissage implies that restriction of scalars induces an equivalence $K(k) \simeq K({\cal P}(k), w{\cal P}(k)) \to K({\cal P}({\cal O}_K)^{v{\cal P}({\cal O}_K)}, w{\cal P}({\cal O}_K))$. This identifies the upper horizontal cofiber sequence of \eqref{eq:waldcof} with the usual cofiber sequence $K(k) \xrightarrow{K(p_*)} K({\cal O}_K) \xrightarrow{} K(K)$ induced by the Karoubi sequence \[{\rm Perf}({\cal O}_K)^{\pi-{\rm nil}} \xrightarrow{} {\rm Perf}({\cal O}_K) \xrightarrow{} {\rm Perf}(K).\] We now discuss the lower horizontal sequence in \eqref{eq:waldcof}. The terms are defined in \cite[Section 1.2]{HM03}, in which Hesselholt--Madsen associate to any Waldhausen category enriched in abelian groups a spectrally enriched category, of which we can form the topological Hochschild homology using the usual cyclic nerve. This variant of ${\rm THH}$ comes with a version of the fibration theorem \cite[Theorem 1.3.11]{HM03} and d\'evissage \cite[Theorem 1]{Dun98} (see Remark \ref{rem:devissageweird}), but no approximation theorem. In conclusion, the diagram \eqref{eq:waldcof} takes the form  \begin{equation}\label{eq:waldcof2}\begin{tikzcd} K(k) \ar{r}{K(p_*)} \ar{d}{{\rm tr}} & K({\cal O}_K) \ar{r} \ar{d}{{\rm tr}} & K(K) \ar{d}{{\rm tr}} \\ {\rm THH}(k) \ar{r}{{\rm THH}(p_*)} & {\rm THH}({\cal O}_K) \ar{r} & {\rm THH}({\cal O}_K | K),\end{tikzcd}\end{equation} where ${\rm THH}({\cal O}_K | K) := {\rm THH}({\cal P}({\cal O}_K), v{\cal P}({\cal O}_K))$.

\begin{remark}\label{rem:devissageweird} While the bottom horizontal cofiber sequence in \eqref{eq:waldcof2} comes to life from d\'evissage, it is not the case that the localizing invariant ${\rm THH}$ satisfies d\'evissage. We now flesh out this potentially confusing discrepancy. Throughout this remark, we sweep some details of the setup of \cite{BM20} under the rug so that the point of the discussion is not lost in technicalities. 

Any stable $\infty$-category ${\cal C}$ is naturally enriched in spectra, leading to the usual definition of ${\rm THH}({\cal C})$ as the realization of the cyclic nerve. There is also a different choice of enrichment, as explained in \cite[Section 2.2]{BM20}: One may define ${\rm map}_{{\cal C}}^{\Gamma}(x, y)$ to be the connective spectrum associated to the Segal $\Gamma$-space with $n$th space ${\rm Map}_{\cal C}(x, \bigoplus_{i = 1}^n y)$. Forming the cyclic nerve with respect to the resulting spectral enrichment, one obtains an alternative definition ${\rm THH}^\Gamma({\cal C})$. This choice is in fact rather natural; for example, following this recipe in the context of exact categories gives exactly the spectral enrichment considered by e.g.\ Dundas--McCarthy \cite{DM96} and Hesselholt--Madsen \cite{HM03} by \cite[Example 2.13]{BM20}. 

If $R$ is a \emph{connective} ${\Bbb E}_{\infty}$-ring, then ${\rm THH}^\Gamma({\rm Perf}(R)) \xrightarrow{\simeq} {\rm THH}({\rm Perf}(R))$ by \cite[Theorem 2.56]{BM20}. But the categories that naturally appear when studying the localization properties for localizing invariants very quickly force one out of the connective setting: For example, in the diagram \begin{equation}\label{eq:nonconndiag}\begin{tikzcd}{\rm THH}^\Gamma({\rm Perf}({\Bbb Z})^{p - {\rm nil}}) \ar{r} \ar{d} & {\rm THH}^\Gamma({\rm Perf}({\Bbb Z})) \ar{d}{\simeq} \\ {\rm THH}({\rm Perf}({\Bbb Z})^{p - {\rm nil}}) \ar{r} & {\rm THH}({\rm Perf}({\Bbb Z})),\end{tikzcd}\end{equation} only the right-hand vertical map is an equivalence. Indeed, Schwede--Shipley \cite[Theorem 3.3]{SS03} applies to obtain an equivalence ${\rm Perf}({\Bbb Z})^{p - {\rm nil}} \simeq {\rm Perf}({\rm End}_{{\Bbb Z}}({\Bbb Z}/p))$, which are compact modules over a non-connective ring spectrum with homotopy groups as ${\Bbb Z}/p$ in degrees $-1$ and $0$. In this situation, \cite[Theorem 1]{Dun98} applies to identify \[{\rm THH}({\Bbb F}_p) \xleftarrow{\simeq} {\rm THH}^\Gamma({\rm Perf}({\Bbb F}_p)) \xrightarrow{\simeq} {\rm THH}^\Gamma({\rm Perf}({\Bbb Z})^{p - {\rm nil}}).\] The resulting cofiber of the top horizontal map in \eqref{eq:nonconndiag} is modelled as the ${\rm THH}^\Gamma$ of a Waldhausen category in \cite{HM03}, which differs from ${\rm THH}({\Bbb Z}[1/p])$. 
\end{remark}

\subsection{Constructing the comparison map} We now have two models for the cofiber of the transfer map ${\rm THH}(p_*) \colon {\rm THH}(k) \to {\rm THH}({\cal O}_K)$, namely the abstractly defined cofiber ${\rm THH}({\cal O}_K | \pi)$ and the Hesselholt--Madsen term ${\rm THH}({\cal O}_K | K)$. By virtue of its definition in terms of the Waldhausen category $({\cal P}({\cal O}_K), v{\cal P}({\cal O}_K))$, the symmetric monoidal structure of ${\cal P}({\cal O}_K)$ endows ${\rm THH}({\cal O}_K | K)$ with the structure of an ${\Bbb E}_{\infty}$-algebra in cyclotomic spectra. We construct a map \begin{equation}\label{eq:tohm}{\rm THH}({\cal O}_K, \langle \pi \rangle) \to {\rm THH}({\cal O}_K | K)\end{equation} of ${\rm THH}({\cal O}_K)$-modules in cyclotomic spectra by base-changing the map \[{\Bbb S}[B^{\rm rep}(\langle t \rangle)] \xrightarrow{\simeq} {\rm THH}({\Bbb S}[t] | t) \xrightarrow{t \mapsto \pi} {\rm THH}({\cal O}_K | K)\] of ${\Bbb S}[B^{\rm cyc}(\langle t \rangle)]$-modules in cyclotomic spectra along ${\Bbb S}[B^{\rm cyc}(\langle t \rangle)] \xrightarrow{t \mapsto \pi} {\rm THH}({\cal O}_K)$. 

\begin{proposition}\label{prop:dgrings} The map $\pi_*{\rm THH}({\cal O}_K, \langle \pi \rangle) \xrightarrow{\pi_*\eqref{eq:tohm}} \pi_*{\rm THH}({\cal O}_K | K)$ is one of differential graded rings. 
\end{proposition}

\begin{proof} Since both the source and target of \eqref{eq:tohm} are ${\Bbb E}_{\infty}$-algebras in cyclotomic spectra, and in particular ${\Bbb E}_{\infty}$-algebras in ${\rm Sp}^{BS^1}$, the degree-increasing differential $\delta$ induced by the $S^1$-action exhibits both as differential graded rings \cite[Section 1.4]{Hes96}. We have here used that both ${\rm THH}({\cal O}_K, \langle \pi \rangle)$ and ${\rm THH}({\cal O}_K | K)$ are ${\cal O}_K$-algebras to ensure that the differentials square to zero. Since \eqref{eq:tohm} is a map of cyclotomic spectra, we have that the square \[\begin{tikzcd}\pi_*{\rm THH}({\cal O}_K, \langle \pi \rangle) \ar{r}{\pi_*\eqref{eq:tohm}} \ar{d}{\delta} & \pi_*{\rm THH}({\cal O}_K | K) \ar{d}{\delta}\\ \pi_{* + 1}{\rm THH}({\cal O}_K, \langle \pi \rangle) \ar{r}{\pi_{* + 1}\eqref{eq:tohm}} & \pi_{* + 1}{\rm THH}({\cal O}_K | K) \end{tikzcd}\] commutes. It only remains to show that $\pi_*\eqref{eq:tohm}$ is a map of graded commutative rings. For this, let us consider the morphism \begin{equation}\label{eq:inducedstalg}{\Bbb S}[t] \otimes {\Bbb S}[B\langle t \rangle^{\rm gp}] \cong {\Bbb S}[B^{\rm rep}(\langle t \rangle)] \xrightarrow{\simeq} {\rm THH}({\Bbb S}[t] | t) \xrightarrow{t \mapsto \pi} {\rm THH}({\cal O}_K | K)\end{equation} of ${\Bbb S}[B^{\rm cyc}(\langle t \rangle)]$-modules as an ${\Bbb S}[t]$-module map by restriction of scalars along the inclusion of $0$-simplices ${\Bbb S}[t] \to {\Bbb S}[B^{\rm cyc}(\langle t \rangle)]$. This ${\Bbb S}[t]$-module map is adjoint to the map of spectra \begin{equation}\label{eq:fromreptohm1}{\Bbb S}[B\langle t \rangle^{\rm gp}] \xrightarrow{{\rm dlog}} {\Bbb S}[B^{\rm rep}(\langle t \rangle)] \xrightarrow{\simeq} {\rm THH}({\Bbb S}[t] | t) \xrightarrow{t \mapsto \pi} {\rm THH}({\cal O}_K | K).\end{equation} By Proposition \ref{prop:dlogcomp}, this is equivalent to the composition \[{\Bbb S}[B\langle t \rangle^{\rm gp}] \xrightarrow{t \mapsto \pi} {\Bbb S}[B{\rm GL}_1(K)] \xrightarrow{} K(K) \xrightarrow{{\rm tr}} {\rm THH}({\cal O}_K | K)\] of (homotopy) commutative ring spectra. Since \eqref{eq:inducedstalg} is defined by extension of scalars along the map ${\Bbb S}[t] \to {\Bbb S}[B^{\rm cyc}(\langle t \rangle)] \xrightarrow{t \mapsto \pi} {\rm THH}({\cal O}_K) \xrightarrow{} {\rm THH}({\cal O}_K | K)$ of homotopy commutative ring spectra, it is a map of homotopy commutative ring spectra. Hence this is also the case for the base-change \eqref{eq:tohm} along the multiplicative map ${\rm THH}({\Bbb S}[t]) \xrightarrow{t \mapsto \pi} {\rm THH}({\cal O}_K)$, and in particular $\pi_*\eqref{eq:tohm}$ is a map of graded commutative rings, as claimed. 
\end{proof}

\begin{remark}\label{rem:awkward} The proof does \emph{not} show that \eqref{eq:tohm} is an equivalence of ${\Bbb E}_{\infty}$-algebras in cyclotomic spectra. While we strongly expect that this is the case, we abandon the $S^1$-action as soon as we make reference to the non-equivariant map ${\Bbb S}[t] \xrightarrow{} {\Bbb S}[B^{\rm cyc}(\langle t \rangle)]$. In our attempts to address this issue, the core difficulty lies with a lack of understanding of the mapping space ${\rm Map}_{{\rm CAlg}^{BS^1}}({\Bbb S}[B^{\rm rep}(\langle t \rangle)], {\rm THH}({\cal O}_K | K))$ and in particular its relation to the mapping space ${\rm Map}_{{\rm CAlg}^{BS^1}}({\Bbb S}[B\langle t \rangle^{\rm gp}], {\rm THH}({\cal O}_K | K))$. These difficulties are in stark contrast to the non-logarithmic situation, where there is an equivalence of mapping spaces \[{\rm Map}_{{\rm CAlg}^{BS^1}}({\Bbb S}[B^{\rm cyc}(\langle t \rangle)], {\rm THH}({\cal O}_K | K)) \simeq {\rm Map}_{{\rm CAlg}}({\Bbb S}[t], {\rm THH}({\cal O}_K | K))\] by McClure--Schw\"anzl--Vogt \cite{MSV97}. 

Nonetheless, one may use ${\rm THH}({\cal O}_K, \langle \pi \rangle)$ as a model to reproduce the spectacular result \cite[Theorem C]{HM03} relating logarithmic ${\rm TR}$ to the log de Rham--Witt complex. By inspection, one identifies \cite[p.\ 43]{HM03} as the crucial point in the argument of \emph{loc.\ cit.}\ where an explicit property of the transfer map ${\rm THH}(p_*)$ is used (as opposed to merely having a cofiber sequence with the correct terms). A suitable replacement for this property is the following: For a finite extension $K \subset L$ with ramification index $e_{L / K}$, the diagram \[\begin{tikzcd}[row sep = small]{\rm THH}({\cal O}_K, \langle \pi_K \rangle) \ar{d} \ar{r}{\partial^{\rm rep}} & {\rm THH}({\cal O}_K / \pi_K)[1] \ar{d}{{\rm THH}(i)[1] \cdot e_{L / K}} \\  {\rm THH}({\cal O}_L, \langle \pi_L \rangle) \ar{r}{\partial^{\rm rep}} & {\rm THH}({\cal O}_L / \pi_L)[1]\end{tikzcd}\] commutes, and recovers the diagram \[\begin{tikzcd}[row sep = small]\Omega^1_{({\cal O}_K, \langle \pi_K \rangle)} \ar{d} \ar{r}{} & {\cal O}_K / \pi_K \ar{d}{i \cdot e_{L / K}} \\  \Omega^1_{({\cal O}_L, \langle \pi_L \rangle)} \ar{r}{} & {\cal O}_L / \pi_L\end{tikzcd}\] on $\pi_1$, where $i \colon {\cal O}_K / \pi_K \to {\cal O}_L / \pi_L$ denotes the induced map of residue fields. Using this, one can mimic the argument of \cite{HM03} with ${\rm THH}({\cal O}_K, \langle \pi_K \rangle)$ in place of ${\rm THH}({\cal O}_K | K)$. We leave the details to the interested reader. 
\end{remark}

\subsection{The log differential graded ring $\pi_*{\rm THH}({\cal O}_K | K)$} So far, we have constructed an explicit equivalence ${\rm THH}({\cal O}_K, \langle \pi \rangle) \xrightarrow{\simeq} {\rm THH}({\cal O}_K | K)$ of ${\rm THH}({\cal O}_K)$-modules in cyclotomic spectra. By Proposition \ref{prop:dgrings}, the induced map \begin{equation}\label{eq:logdiffrings}\pi_*{\rm THH}({\cal O}_K, \langle \pi \rangle) \xrightarrow{\cong} \pi_*{\rm THH}({\cal O}_K | K)\end{equation} is an isomorphism of differential graded rings. 

We now aim to show that the map \eqref{eq:logdiffrings} preserves one final piece of structure, namely that of a log differential graded ring (cf.\ Definition \ref{def:logdiff}). We determined this structure on $\pi_*{\rm THH}({\cal O}_K, \langle \pi \rangle)$ in Proposition \ref{prop:logthhlogdiff}, and we will now recall from \cite[Proof of Proposition 2.3.1]{HM03} how it comes to life on $\pi_*{\rm THH}({\cal O}_K | K)$. 

We first observe that Hesselholt--Madsen phrase the log differentials as an invariant of log rings (as opposed to pre-log rings), and so they are concerned with $\Omega^1_{({\cal O}_K, {\cal O}_K \cap {\rm GL}_1(K))}$. By \cite[(1.7)]{Kat89}, the map $\Omega^1_{({\cal O}_K, \langle \pi \rangle)} \xrightarrow{} \Omega^1_{({\cal O}_K, {\cal O}_K \cap {\rm GL}_1(K))}$ is an isomorphism. As explained in the last paragraph of \cite[p.\ 37]{HM03}, one defines the morphism ${\rm dlog} \colon {\cal O}_K \cap {\rm GL}_1(K) \to \pi_1{\rm THH}({\cal O}_K | K)$ as a composite \[\begin{tikzcd}{\cal O}_K \cap {\rm GL}_1(K) \ar[bend right = 5 mm]{rrd}{{\rm dlog}} \ar{r} & \pi_1B({\cal O}_K \cap {\rm GL}_1(K)) \ar{r}{\pi_1\widetilde{{\rm det}}}  & \pi_1K({\cal P}({\cal O}_K), v{\cal P}({\cal O}_K)) \ar{d}{{\rm tr}} \\ \vspace{10 mm} & &  \pi_1{\rm THH}({\cal O}_K | K),\end{tikzcd}\] where $({\cal P}({\cal O}_K), v{\cal P}({\cal O}_K))$ is the Waldhausen category of perfect ${\cal O}_K$-modules with weak equivalences those maps that become equivalences after base-change along ${\cal O}_K \to K$, and ${\rm tr}$ is the trace map to the ${\rm THH}$ of this Waldhausen category (which is how Hesselholt--Madsen define ${\rm THH}({\cal O}_K | K)$). 

We proceed to describe the map $\widetilde{{\rm det}}$. For any object $X$ of a Waldhausen category ${\cal C}$, we may define a map \[\widetilde{\rm det} \colon \Sigma^\infty B{\rm Aut}(X) \to K({\cal C}),\] as explained in \cite[p.\ 37]{HM03}. Here ${\rm Aut}(X)$ is defined to be the \emph{endo}morphisms in the category of weak equivalences in ${\cal C}$. In the example at hand we set $X = {\cal O}_K[0]$, so that ${\rm Aut}(X) = {\cal O}_K \cap {\rm GL}_1(K)$ since weak equivalences are checked after base-change along ${\cal O}_K \to K$. In the example where ${\cal C} = {\cal P}(K)$ is the category of perfect $K$-modules (with its usual weak equivalences) and $X = K[0]$, then ${\rm Aut}(X) = {\rm GL}_1(K)$ and the map $\widetilde{{\rm det}}$ is the usual map $\Sigma^\infty {\rm BGL}_1(K) \to {\Bbb S}[{\rm BGL}_1(K)] \to K(K)$. 

\begin{proof}[Proof of Theorem \ref{thm:hmcomp}] It only remains to prove that the map \eqref{eq:logdiffrings} is a map of log differential graded rings. Let us write ${\cal C} := ({\cal P}({\cal O}_K), v{\cal P}({\cal O}_K))$ for brevity. We contemplate the diagram  \[\begin{tikzcd}[column sep = small]{\cal O}_K \cap {\rm GL}_1(K) \ar{r} &  \pi_1B({\cal O}_K \cap {\rm GL}_1(K)) \ar{d} \ar{r}{\pi_1\widetilde{{\rm det}}} & \pi_1K({\cal C}) \ar{r}{\pi_1{\rm tr}} \ar{d}{\cong}& \pi_1{\rm THH}({\cal O}_K | K) \\ \vspace{10 mm} & \pi_1{\rm BGL}_1(K) \ar{r}{\pi_1\widetilde{\rm det}} & \pi_1K(K) \ar[swap]{ur}{\pi_1{\rm tr}} \\ \langle t \rangle \ar{uu}{t \mapsto \pi_K} \ar{r} & \pi_1B\langle t \rangle^{\rm gp} \ar{u}{t \mapsto \pi_K} \ar{rr}{\rm dlog} & & \pi_1{\rm THH}({\cal O}_K, \langle \pi \rangle) \ar[swap]{uu}{\pi_1\eqref{eq:tohm}} \end{tikzcd}\] of commutative monoids, where the isomorphism is provided by Waldhausen's approximation theorem \cite[Theorem 1.6.7]{Wal85}. We claim that this commutes by construction: For the left-hand square this is clear. For the rest, we observe that the composition \[\pi_1B\langle t \rangle^{\rm gp} \xrightarrow{{\rm dlog}} \pi_1{\rm THH}({\cal O}_K, \langle \pi \rangle) \xrightarrow{\pi_1\eqref{eq:tohm}} \pi_1{\rm THH}({\cal O}_K | K)\] is defined to be $\pi_1$ of the composition \eqref{eq:fromreptohm1}, and so the lower, right-hand square\footnote{With corners $\pi_1B\langle t \rangle^{\rm gp}$, $\pi_1{\rm BGL}_1(K)$, $\pi_1{\rm THH}({\cal O}_K, \langle \pi \rangle)$, and $\pi_1{\rm THH}({\cal O}_K | K)$.} commutes by Proposition \ref{prop:dlogcomp}. The square involving $\widetilde{{\rm det}}$ commutes, and the triangle involving the trace commutes by construction. This concludes the proof. 
\end{proof}

\end{document}